\documentclass[12pt]{amsart}
\usepackage{amsmath, amsthm, amssymb, eqnarray}
\usepackage{accents}
\usepackage{tikz-cd}
\usetikzlibrary{calc}

\input xy
\xyoption{all}

\usepackage{enumitem} 
\usepackage{amsfonts, calligra, amscd, bbm, eucal, latexsym, extarrows, mathrsfs}
\usepackage[utf8]{inputenc}
\usepackage[T1]{fontenc}
\usepackage[bookmarks=false]{hyperref}

\usepackage{fullpage}
\usepackage[all,cmtip]{xy}
\usepackage{microtype}
\usepackage{cancel}
\usepackage{fourier}
\usepackage{graphicx}
\usepackage{soul}

\usepackage[skip=3pt plus1pt, indent=20pt]{parskip}

\newcommand{\nref}[2]{\hyperref[#1]{\ref*{#1}$_{#2}$}}

\usepackage{mathtools}

\theoremstyle{plain}
\newtheorem{theorem}{Theorem}[section]
\newtheorem{proposition-definition}{Proposition/Definition}[section]
\newtheorem{theorem-definition}{Theorem/Definition}[section]

\newtheorem*{theoremA}{Theorem A}
\newtheorem*{theoremB}{Theorem B}
\newtheorem*{theoremC}{Theorem C}
\newtheorem*{proposition-definition-intro}{Proposition/Definition}
\newtheorem*{definition-intro}{Definition}
\newtheorem*{cor-intro}{Corollary}
\newtheorem*{prop-intro}{Proposition}
\newtheorem{proposition}[theorem]{Proposition}		
\newtheorem{corollary}[theorem]{Corollary}
\newtheorem{lemma}[theorem]{Lemma}
\newtheorem*{conjecture}{Conjecture}
\newtheorem*{conjecture-intro}{Conjecture}

\newtheorem{construction-definition}[theorem]{Construction/Definition}

\newtheorem{problem-Deligne}{Problem}

\theoremstyle{remark}

\newtheorem*{remark-intro}{Remark}

\newcommand{\ov}{\overline}

\newcommand{\CBbb}{\mathbb C}

\newcommand{\NBbb}{\mathbb N}

\newcommand{\PBbb}{\mathbb P}
\newcommand{\QBbb}{\mathbb Q}
\newcommand{\RBbb}{\mathbb R}

\newcommand{\ZBbb}{\mathbb Z}

\newcommand{\Fcal}{\mathcal F}

\newcommand{\Hcal}{\mathcal H}

\newcommand{\Ocal}{\mathcal O}

\newcommand{\Xcal}{\mathcal X}

\providecommand{\vol}{\operatorname{vol}}

\providecommand{\Gr}{\operatorname{Gr}}

\DeclareMathOperator{\Mil}{Mil}

\DeclareMathOperator{\SheafHom}{\mathscr{H}\text{\kern -3pt {\calligra\large om}}\,}

\DeclareMathOperator{\tr}{tr}
\DeclareMathOperator{\Real}{Re}

\numberwithin{equation}{section}

\begin{document}
\setcounter{tocdepth}{1}
\title{The spectral genus of an isolated hypersurface singularity and a conjecture relating to the Milnor number}
\author{Dennis Eriksson}
\author{Gerard Freixas i Montplet}

\address{Dennis Eriksson \\ Department of Mathematics \\ Chalmers University of Technology and  University of Gothenburg}
\email{dener@chalmers.se}

\address{Gerard Freixas i Montplet \\ CNRS -- Centre de Math\'ematiques Laurent Schwartz - \'Ecole Polytechnique - Institut Polytechnique de Paris}
\email{gerard.freixas@polytechnique.edu}




\begin{abstract}
In this paper, we introduce the notion of spectral genus $\widetilde{p}_{g}$ of a germ of an isolated hypersurface singularity  $(\CBbb^{n+1}, 0) \to (\CBbb, 0)$, defined as a sum of small exponents of monodromy eigenvalues. The number of these is equal to the geometric genus $p_{g}$, and hence $\widetilde{p}_g$ can be considered as a secondary invariant to it. We then explore a secondary version of the Durfee conjecture on $p_{g}$, and we predict an inequality between $\widetilde{p}_{g}$ and the Milnor number $\mu$, to the effect that $$\widetilde{p}_g\leq\frac{\mu-1}{(n+2)!}.$$ We provide evidence by confirming our conjecture in several cases, including homogeneous singularities and singularities with large Newton polyhedra, and quasi-homogeneous or irreducible curve singularities. We also show that a weaker inequality follows from Durfee's conjecture, and hence holds for quasi-homogeneous singularities and curve singularities.

Our conjecture is shown to relate closely to the asymptotic behavior of the holomorphic analytic torsion of the sheaf of holomorphic functions on a degeneration of projective varieties, potentially indicating deeper geometric and analytic connections. 
\end{abstract}

\maketitle

\setcounter{tocdepth}{1}
\tableofcontents

\section{Introduction}

\subsection{}\label{subsec:Durfee-intro} In 1978, Durfee conjectured an inequality between the geometric genus $p_g$ and the Milnor number of complete intersection isolated surface singularities, in \cite{Durfee}. While the initial expectation does not hold in such generality, the conjecture was later extended in \cite{KSaito:distribution} by K. Saito to a conjecture about general germs of  $n$-dimensional isolated hypersurface singularities\footnote{In this paper, by a singularity, we mean a non-regular point.}, with $n\geq 2$, defined by germs of holomorphic functions $f: (\CBbb^{n+1},0) \to (\CBbb,0)$, as an inequality 
\begin{equation}\label{eq:Durfee-type}
    p_{g}<\frac{\mu}{(n+1)!}.
\end{equation}
This generalized inequality appeared naturally in his investigation of the distribution of the spectrum of the semi-simple part of the monodromy, in the same article. Particular cases and variants of the Durfee-type conjecture \eqref{eq:Durfee-type} have since been established by N\'emethi \cite{Nemethi-selecta-1, Nemethi-selecta-2} (suspension-type surfaces), Yau--Zhang \cite{Yau-Zhang} (quasi-homogeneous singularities) and  Kerner--N\'emethi \cite{Kerner-Nemethi-2} (generic singularities with large Newton diagram).

\subsection{} By work of M. Saito and Steenbrink \cite{MSaitogenus, SteenbrinkMixedAssociated}, the geometric genus of an isolated hypersurface singularity $f=0$ is related to the cohomology of the Milnor fiber and its mixed Hodge structure by $p_g = \dim \Gr^n_F H^n(\Mil_f)$. We define the \emph{spectral genus of the singularity} (cf. \textsection\ref{section:genus} below) as 
\begin{displaymath}
    \widetilde{p}_g = \sum \lambda_j,
\end{displaymath}
where the sum is over the rational numbers $\lambda_j\in [0,1)$ such that $\exp(2\pi i \lambda_j)$ is an eigenvalue of the semi-simple part of the monodromy acting on $\Gr^n_F H^n(\Mil_f)$. It is a sort of secondary invariant of the geometric genus, and is in particular zero for rational or, equivalently, canonical singularities. One may wonder if there is a corresponding secondary version of the Durfee-type conjecture. In light of this, we propose the following:
\begin{conjecture}
 For an $n$-dimensional isolated hypersurface singularity, with $n\geq 1$, we have an inequality:
 \begin{enumerate}
     \item (Weak form) 
     \begin{displaymath}
      \widetilde{p}_g <  \frac{ \mu }{(n+2)!} .
   \end{displaymath}
     \item (Strong form) 
     \begin{displaymath}
          \widetilde{p}_g  \leq \frac{ \mu -1}{(n+2)!} .
     \end{displaymath}
 \end{enumerate}
\end{conjecture}
The inequality in the strong form is clearly satisfied for rational singularities, since then the spectral genus vanishes and $\mu\geq 1$. It is attained in the particular case of ordinary double points, since the Milnor number is moreover one. It seems natural to wonder if this is the only situation where the inequality is attained. Moreover, if K. Saito's conjectures in \cite{KSaito:distribution} about the distribution of monodromy eigenvalues hold, then the above inequalities are essentially optimal. Namely, it implies (cf. Proposition \ref{prop:consequence-KSaito}) that 
\begin{displaymath}
    \frac{\widetilde{p}_g}{\mu}  \to \frac{1}{(n+2)!}
\end{displaymath}
as the singularities get worse.

\subsection{} Our first main contribution is the following partial confirmation: 
\begin{theoremA}
The strong form of the conjecture is true in the following cases: 
    \begin{enumerate}[label=(\roman*)]
        \item Homogeneous singularities in arbitrary dimension. 
        \item Quasi-homogeneous curve singularities. 
        \item Irreducible curve singularities.
        \item Generic and convenient singularities with large Newton polyhedra in arbitrary dimension. 
    \end{enumerate}
\end{theoremA}
For the proofs, see the corresponding sections in the article, relying in one way or another on the fact that the spectrum for non-degenerate singularities can be described in terms of Newton polyhedra. The case of irreducible curve singularities covers also degenerate cases. 

We remark that other singularities have been tested with the software \textsc{Singular} \cite{DGPS}, allowing instantaneous computation of all involved invariants. In fact, the conjecture was initially based on an extensive numerical investigation, utilizing this software, of the expressions $\mu/6-\widetilde{p}_g$ for curve singularities and $\mu/24-\widetilde{p}_g$ for surface singularities. 
\subsection{} The extension of the conjecture to higher dimensions was motivated by the asymptotic behavior of the holomorphic analytic torsion, which is defined in terms of regularized determinants of Dolbeault--Laplace operators. It is denoted by $\tau$, and it is a strictly positive real number. More precisely, the rate of vanishing or blowing-up of $\tau(\Ocal_{\Xcal_{t}})$ for a degenerating family of projective varieties $\Xcal \to\Delta$ with isolated singularities, is essentially captured by the expression $(-1)^n (\mu/(n+2)! - \widetilde{p}_g)$ (cf. Corollary \ref{cor:asymptoticholomorphictorsion}). This is related to the work of Yoshikawa on the singularities of the Quillen metric \cite{ yoshikawa2, yoshikawa}. See also \cite{cdg} for an interpretation of Yoshikawa's results in terms of intersection theory. 

The weak form of the conjecture is connected to the question whether the function
\begin{equation}\label{eq:analytictorsion}
    X \mapsto \tau(\Ocal_{X})^{(-1)^n}
\end{equation}
 extends continuously over the space of hypersurfaces in $\PBbb^{n+1}$. Such a property would imply the weak form of the conjecture (cf. \textsection\ref{subsec:asymptoticholomorphictorsion} to \textsection\ref{subsec:reformulation}). This is based on the fact that the function in \eqref{eq:analytictorsion}, if continuous, vanishes on the locus of ordinary double points, because the weak form of the conjecture is known then, and any isolated singularity is a limit of ordinary double points. The strong form of the conjecture even allows to predict the order of vanishing along the discriminant locus. This formulation of the problem possibly invites analytic techniques of global nature to study the conjecture, which is of local type.  

\subsection{} By means of a suspension trick, that J. I. Burgos Gil generously shared with us, we prove in Proposition \ref{prop:burgos} that the weak form of the conjecture is a consequence of the Durfee-type conjecture \eqref{eq:Durfee-type}. This leads to the second main result of this article:

\begin{theoremB}
The weak form of the conjecture holds in the following cases:

\begin{enumerate}
    \item Plane curves singularities, ie. when $n=1$ in general.
    \item Quasi-homogeneous singularities in arbitrary dimension.
\end{enumerate}
\end{theoremB}
The proof, presented in \textsection\ref{subsec:proof-thmB}, is deduced from the work of N\'emethi on the Durfee conjecture for suspensions of curves, and the work of Yau--Zhang for quasi-homogeneous singularities, both recalled in \textsection\ref{subsec:Durfee-intro} above. 

As for the strong form of the conjecture, it would be interesting to relate it to a Durfee-type bound as well, but we were not able to derive it from the existing refinements, such as those in \cite{Kerner-Nemethi-2, Xu-Yau-fourdimensionaltetrahedra, Yau-Zhang}.

\subsection{} Exploiting the relationship between our conjecture and the asymptotic behavior of the analytic torsion stated in Corollary \ref{cor:asymptoticholomorphictorsion}, Theorem A and Theorem B have the following concrete application to determinants of Laplacians, which appears to be new.
\begin{theoremC}
Let $\Xcal \to\Delta$ a degeneration of compact Riemann surfaces, with isolated singularities in the central fiber $\Xcal_{0}$ and with $\Xcal$ smooth Kähler. Then, the determinant of the Laplace--Beltrami operator on  $\Xcal_{t}$ satisfies
\begin{displaymath}
    \det\Delta_{\Xcal_{t}}\to 0,\quad\text{as}\quad t\to 0.
\end{displaymath}
Furthermore, if the singularities of $\Xcal_{0}$ are locally irreducible or quasi-homogeneous, then for every $\varepsilon>0$  we have
\begin{displaymath}
    \det\Delta_{\Xcal_{t}}=O\left(|t|^{m/3-\varepsilon}\right),
\end{displaymath}
where $m$ is the number of singular points in $\Xcal_{0}$.
\end{theoremC}
The strong form of the conjecture suggests that the second part of Theorem C should hold for general projective degenerations with isolated singularities. 

For degenerations of compact hyperbolic Riemann surfaces, an analogue of Theorem C above is known too. This is contained in the work of Wolpert on degenerations of Selberg zeta functions \cite{Wolpert:Selberg}, and can also be derived from  \cite{GFM:ARR}. In this setting, it is enough to consider stable degenerations, in which case the singular fiber is endowed with a complete metric. This is somewhat opposite to the situation treated in Theorem C, for which it seems not possible to perform a semi-stable reduction, and moreover the singular fiber carries an incomplete metric. 

\subsection{} In future work we aim to apply and extend part of this discussion to the asymptotic behavior of the BCOV invariant of Calabi--Yau varieties, introduced in dimension 3 in \cite{FLY} and in general dimension in \cite{cdg2}. Some instances of such asymptotics were used to establish new cases of genus one mirror symmetry in \cite{cdg3}. Further qualitative discussions of the same type could provide even further cases and new insights into mirror symmetry phenomena.



\section{Invariants of isolated singularties}

In this section, to set up notation and for the convenience of the reader, we recall some classical invariants of isolated hypersurface singularities.  We will be considering the germ of a holomorphic function $f: (\CBbb^{n+1}, 0) \to (\CBbb, 0)$ defining an isolated hypersurface singularity at the origin, and we will suppose that $n \geq 1$. Somewhat abusively, sometimes we will simply write $f=0$.

\subsection{} It is convenient to study these isolated hypersurface singularities through the cohomology of the Milnor fiber $H^n(\Mil_f)$. Note that the (reduced) cohomology of the Milnor fiber vanishes in other degrees. Its dimension over $\CBbb$ is the  Milnor number of the singularity: 

\begin{displaymath}
    \mu = \mu_f=\dim \CBbb \{x_0, \ldots, x_n\}/\left( \partial f/\partial x_0, \ldots, \partial f/\partial x_n \right).
\end{displaymath}
The cohomology group $H^n(\Mil_f)$ is equipped with a canonical mixed Hodge structure by \cite{Steenbrink-mixedonvanishing}. We denote the corresponding Hodge and weight filtrations by $F$ and $W$, respectively. Moreover, the semi-simple part of the monodromy, denoted by $T_s$, acts on this mixed Hodge, in the sense that it preserves the weight and Hodge filtrations \cite[Theorem 4.1]{Steenbrink-mixedonvanishing}. 

\subsection{}\label{section:genus} Consider the genus of the singularity $V = \{f=0\}$:
\begin{equation}\label{eq:genus}
    p_g = \dim \Gr^n_F H^n(\Mil_f) =    \begin{cases} \dim R^{n-1} \pi_\ast \Ocal_{X} & \mbox{if } n > 1, \\ \dim \pi_{\ast} \Ocal_X/\Ocal_V & \mbox{if } n=1, \end{cases}
\end{equation}
where $\pi: X \to V$ is a desingularization. The equality of the various quantities is proven in \cite[Theorem 1]{MSaitogenus} and \cite[Proposition 2.13]{SteenbrinkMixedAssociated}. In particular, $p_{g}=0$ if $V$ has rational singularities. In our setting, since $V$ is Gorenstein, this is equivalent to saying that $V$ has canonical singularities \cite{Elkik:singularites-canoniques}.

Recall from the introduction that we likewise define the \emph{spectral genus of the singularity}, as the expression 
\begin{equation}\label{eq:definition-spectral-genus}
    \widetilde{p}_{g}=  \sum \lambda_j,
\end{equation}
where the sum is over all $0 < \lambda_j < 1$ such that $\exp(2\pi i \lambda_j)$ is an eigenvalue of $T_s$ acting on $\Gr^n_F H^n(\Mil_f)$. More compactly, \eqref{eq:definition-spectral-genus} can be recast as 
\begin{equation}\label{eq:spectralgenus}
    \widetilde{p}_g = \frac{1}{2\pi i } \tr \left(\log (T_s) \mid { \Gr^n_F H^n(\Mil_f) }\right),
\end{equation}
where $\log$ is the branch of the logarithm whose imaginary part lies in $[0,2\pi)$. 

\subsection{}\label{subsec:conventions-spectrum} We next elaborate on the relationship between the spectral genus and the spectrum of an isolated hypersurface singularity. We first follow Steenbrink's presentation in \cite[Section 2]{Steenbrink-semicontinuity} and \cite[Section 1]{Steenbrink-asterisque}, and we adopt his conventions. See also \cite[Section 12.1.3]{Peters-Steenbrink}. The spectral numbers are associated to the triple $(H^{n}(\Mil_{f}), F^{\bullet}, T_{s})$. These are rational numbers $\alpha$, given with multiplicities, uniquely determined by the following conditions:
\begin{enumerate}  
    \item $\exp(-2\pi i\alpha)$ is an eigenvalue of $T_{s}$ acting on $\Gr^{p}_{F}H^{n}(\Mil_{f})$, for some $p=0,\ldots,n$. 
    \item For $\alpha$ and $p$ as in the first point, $p=[n-\alpha]$. Equivalently, we take $n-p-1<\alpha\leq n-p$.
    \item The multiplicity of $\alpha$ is the multiplicity of $\exp(-2\pi i\alpha)$.
\end{enumerate}
Denote by $\lbrace \alpha_{j}\rbrace_{j=1,\ldots, \mu}$ the collection of spectral numbers, with multiplicities. This collection is invariant under $\alpha\mapsto n-1-\alpha$. The spectral numbers hence belong to the interval $(-1,n)$. With this understood, we see that the spectral genus $\widetilde{p}_{g}$ can be expressed as
\begin{displaymath}
    \widetilde{p}_{g}=-\sum_{\alpha_{j}< 0}\alpha_{j}.
\end{displaymath}
By the symmetry of the spectral numbers with respect to $\alpha\mapsto n-1-\alpha$, we can equivalently write
\begin{equation}\label{eq:spectral-genus-Steenbrink-GrF0}
    \widetilde{p}_{g}=\sum_{j}\lambda_{j}^{\prime},
\end{equation}
where the sum is now over rationals $0<\lambda_{j}^{\prime}<1$ such that $\exp(-2\pi i\lambda_{j}^{\prime})$ is an eigenvalue of $T_{s}$ acting on $\Gr_{F}^{0}H^{n}(\Mil_{f})$. This is to be compared with \eqref{eq:definition-spectral-genus}.

Some authors shift the spectral numbers by one, so that they are given in the form $\alpha^{\prime}_{j}=\alpha_{j}+1$ and belong to $(0,n+1)$. With this convention, we have
\begin{equation}\label{eq:spgenusMSaito}
    \widetilde{p}_{g}=\sum_{\alpha_{j}^{\prime}< 1}(1-\alpha_{j}^{\prime}).
\end{equation}
This convention appears in the works of M. Saito \cite{MSaito-Newton, MSaito:Hertling}, applied below. Depending on the context, one convention may be more adapted than the other, and we will use both.

\subsection{} The Milnor number $\mu$ and the geometric genus $p_{g}$ depend only on the fiber $V = f^{-1}(0)$, but not on the chosen deformation \cite[Theorem 2.9]{Steenbrink-semicontinuity}. The spectral genus depends on the whole germ $f$. Nevertheless, the spectrum is constant in any deformation of isolated hypersurface singularities with constant $\mu$ \cite[Theorem 2.8]{Steenbrink-semicontinuity}, and hence so is the spectral genus.

\subsection{}\label{subsec:various-spectral-preliminaries} For later use, we recall the definition of the spectral polynomial associated to $f$, which for the spectral numbers taken in $(0,n+1)$ is given by
\begin{equation}\label{eq:spectral-polynomial}
    \mathrm{Sp}_{f}(T)=\sum_{j}T^{\alpha_{j}^{\prime}}\in\ZBbb[T^{\QBbb}].
\end{equation}
We also recall the Thom--Sebastiani property for the spectral numbers \cite[Theorem 7.3]{Varchenko-asymptotic}. If $h\colon (\CBbb^{m+1},0)\to (\CBbb,0)$ defines another isolated hypersurface singularity, with spectral numbers $\beta_{j}^{\prime}$ taken in $(0,m+1)$, then the spectral numbers of $f(x_{0},\ldots,x_{n})+h(y_{0},\ldots,y_{m})$ are given by the sums
\begin{displaymath}
    \alpha_{i}^{\prime}+\beta_{j}^{\prime},\quad\text{for}\ i=1,\ldots,\mu_{f}\ \text{and}\ j=1,\ldots,\mu_{h},
\end{displaymath}
which hence belong to $(0,m+n+2)$.

\subsection{} 

We next review some basic facts on the local and global theories of isolated singularities, in connection with degenerations of Hodge structures. Consider a connected complex manifold $\Xcal$ of dimension $n+1\geq 2$, and a flat, projective morphism $g: \Xcal \to \Delta$, which is a holomorphic submersion outside the origin. We will suppose that the central fiber has at most isolated singularities. If $x_i \in \Xcal_0$ is such a singular point, the germ $(\Xcal, x_i) \to (\Delta, 0)$ is isomorphic to some $f_i: (\CBbb^{n+1}, 0) \to (\CBbb, 0)$ and admits a Milnor fiber $\Mil_{f_i}$. In this setting, we denote by $\mu$ and $\widetilde{p}_{g}$ the sum of the Milnor numbers and spectral genera of the singularities $x_{i}$, respectively. Therefore, there is a decomposition
\begin{displaymath}
    \frac{\mu}{(n+2)!} - \widetilde{p}_g = \sum_i  \left(\frac{\mu_i}{(n+2)!} - \widetilde{p}_{g,i}\right).
\end{displaymath}

\subsection{} The cohomologies of the Milnor fibers of the singularities $x_{i}$ sit in an exact sequence 
\begin{equation}\label{Milnorsequence}
    0 \to H^n(\Xcal_0) \to H^n(\Xcal_t) \to \bigoplus_i H^n(\Mil_{f_i}) \to H^{n+1}(\Xcal_0) \to H^{n+1}(\Xcal_t) \to 0,
\end{equation}
for fixed $t\neq 0$, and for $q \neq n, n+1$ there is an isomorphism
\begin{equation}\label{Milnorsequence1}
    H^q(\Xcal_0) \simeq H^q(\Xcal_t).
\end{equation}
When the $H^k(\Xcal_0)$ are given the canonical mixed Hodge structures of Deligne, and the $H^k(\Xcal_t)$ are given the limit mixed Hodge structures of Schmid, then \eqref{Milnorsequence} is moreover an exact sequence of mixed Hodge structures. It is equivariant with respect to the semi-simple part of the monodromy, $T_s$. For details, we refer to Steenbrink \cite[Section 3.3]{Steenbrink-mixedonvanishing} and Navarro Aznar \cite[Section 14]{Navarro}. Below, we will denote the limit mixed Hodge structure in degree $k$ simply by $(H^{k}_{\lim}, F^{\bullet}, W_{\bullet})$. 

\subsection{}\label{subsec:degeneration} If we start with the germ of an isolated singularity $f : (\CBbb^{n+1},0) \to (\CBbb,0)$, it admits a \emph{good compactification} $g:  \Xcal \to \Delta$. By this, we mean:
\begin{enumerate}
    \item\label{item:degeneration-1} We are given a complex manifold $\Xcal$, and a flat, projective morphism $g: \Xcal \to \Delta$, which is a submersion outside the origin. We refer to $\Xcal\to\Delta$ as a \emph{degeneration} of projective varieties, or simply a projective degeneration.
    \item The special fiber $\Xcal_0$ has only one isolated singularity $x$.
    \item There is an open subset $U$ of $x$ such that $(g|_U,x) \to (\Delta,0)$ is isomorphic to $f$.

\end{enumerate}
By an argument of Brieskorn  \cite[Section {1.1}]{Brieskorn}, such a compactification exists, and one can further suppose that it is given by a family of hypersurfaces in $\PBbb^{n+1}$. This moreover shows that properties such as the positivity of $\frac{\mu}{(n+2)!} - \widetilde{p}_g$ can be studied equivalently for proper families or in the local setting. This will be used to reformulate our conjecture in Proposition \ref{prop:reformulation} below.

\section{Analytic torsion of $\Ocal_X$}\label{sec:analytictorsion}

In this section, we discuss the asymptotic behavior of the holomorphic analytic torsion of the sheaf of holomorphic functions, for a degeneration of projective varieties, and relate it to our conjecture on the spectral genus.

\subsection{} Let $X$ be a compact analytic space and suppose that we are given a holomorphic vector bundle $E$ on $X$. The determinant of the cohomology is the line 
\begin{equation}\label{eq:detcohpoint}
    \lambda(E) = \bigotimes_q \det H^q(X,E)^{(-1)^q}.
\end{equation}
More generally, for a flat proper morphism of complex analytic spaces $\Xcal  \to S$ over a complex analytic manifold $S$, and a vector bundle $E$ on $\Xcal$, there is a line bundle $\lambda(E)$ on $S$, whose fibers over $s \in S$ are given by $\lambda(E|_{\Xcal_s})$, see \cite[Section 4.1]{bismutbost}. It is also referred to as the determinant of the cohomology. If $\Xcal \to S$ is the analytification of an algebraic family, this construction is the analytification of the Knudsen--Mumford determinant \cite{KnudsenMumford}. It has a natural grading, which for the purposes of this article, together with various sign issues, we can ignore. 

\subsection{}\label{subsec:analytictorsion} Let $X$ be a compact Kähler manifold. The analytic torsion of a hermitian vector bundle $E$ on $X$ is defined as a weighted alternating product of determinants of Laplacians, namely
\begin{displaymath}
    \tau(X,E) = \exp \left( \sum (-1)^{q+1} q \zeta'_{0,q}(0)\right)=\prod_{q} (\det\Delta_{\ov{\partial}}^{0,q})^{(-1)^{q}q}.
\end{displaymath}
Here, for $\Real(s)\gg 0$, $\zeta_{0,q}(s)$ is given by
\begin{displaymath}
    \zeta_{0,q}(s) = \sum \frac{1}{\mu_j^s},
\end{displaymath}
whose sum runs over positive eigenvalues $\mu_j$ of the Dolbeault-Laplacian $\Delta_{\overline{\partial}}^{0,q}$ acting on $A^{0,q}(E)$. It depends on both the Kähler metric on $X$ and the hermitian metric on $E$.

\subsection{}\label{sec:Quillendeg} Let $X$ be a Kähler manifold and $E$ a hermitian vector bundle on $X$. The determinant of the cohomology in \eqref{eq:detcohpoint} is equipped with two metrics, the $L^2$-metric and the Quillen metric. The $L^2$-metric  $h_{L^2}$ on $\lambda(E)$ is defined by representing the Dolbeault cohomology groups $H^q(X,E)$ by harmonic forms, and using the natural metric from Hodge theory on $A^{0,q}(E)$-forms. The Quillen metric on $\lambda(E)$ is defined by 
\begin{equation}\label{eq:Quillenmetric}
   h_Q =  h_{L^2} \cdot \tau(X,E),
\end{equation}
where $\tau(X,E)$ is the analytic torsion described in \eqref{subsec:analytictorsion}. 

For a family of Kähler manifolds $\Xcal \to S$ and a hermitian vector bundle $E$, the Quillen metric varies smoothly. If the dimensions $s \mapsto h^q(\Xcal_s, E|_{\Xcal_s})$ are constant on $S$ for all $q$, the $L^2$-metric is also smooth. 

\subsection{}\label{subsec:degeneration-Quillen} Let $\Xcal \to \Delta$ be a degeneration of projective varieties of dimension $n\geq 1$ as in \textsection\ref{subsec:degeneration} \eqref{item:degeneration-1}, with isolated singularities in the central fiber $\Xcal_0$. Suppose that $\Xcal$ is equipped with a K\"ahler metric, and that we are given a hermitian vector bundle $E$ on $\Xcal,$ of rank $e$. Let $\sigma$ be a holomorphic trivialization of the determinant of the cohomology $\lambda(E)$. By the main results of \cite{yoshikawa2, yoshikawa}, we have
\begin{equation}\label{eq:asymptoticQuillen}
    \log\|\sigma \|_{Q}^{2} = \frac{(-1)^{n} } {(n+2)!} \mu \cdot e \cdot  \log|t|^{2} + O(1),\quad\text{as}\quad t\to 0.
\end{equation}

\subsection{}\label{sec:L2deg} Suppose now that $\Xcal\to\Delta$ is as in \textsection\ref{subsec:degeneration-Quillen}, but $\Xcal$ is actually equipped with a Kähler metric, whose associated Kähler form is rational when restricted to smooth fibers. Endow the sheaf of holomorphic functions $\Ocal_{\Xcal}$ with the trivial hermitian metric, induced by the absolute value. 

\begin{proposition}\label{prop:L2deg}
    Let $\sigma$ be a trivialization of $\lambda(\Ocal_{\Xcal})$. Then, for small $t$, 
    \begin{displaymath}
        \log\|\sigma \|^2_{L^2} = (-1)^n \widetilde{p}_g \log|t|^2 +(-1)^{n+1} \beta\log\log|t|^{-1}+O(1),
    \end{displaymath}
where $\beta\geq 0$ is an integer determined by the limit mixed Hodge structure $H^{n}_{\lim}$, and given by
\begin{equation}\label{eq:beta}
    \beta=\sum_{r=1}^{n}r\dim F^{n}\Gr^{W}_{n+r}H^{n}_{\lim}.
\end{equation}
In particular, if $\Xcal_{0}$ has canonical singularities, then $\log\|\sigma\|^{2}_{L^{2}}=O(1)$.
\end{proposition}
\begin{proof}
    First of all, the determinant of the cohomology of the sheaf of holomorphic functions is invariant under blowups in the special fiber, because the total space is smooth and hence has only rational singularities. We can hence suppose the central fiber of $\Xcal \to \Delta$ has normal crossings. In this case, by \cite[Theorem C]{cdg2} we find that 
    \begin{displaymath}
        \log \|\sigma \|_{L^2}^{2} = \left(\sum_{q=0}^{n} (-1)^q \alpha^{0,q} \right) \log|t|^{2} + \left(\sum_{q=0}^{n} (-1)^{q}\beta^{0,q}\right)\log\log|t|^{-1}+O(1),
    \end{displaymath}
    where $\alpha^{0,q}$ is defined as minus (the lower extension of) the logarithm of the semi-simple part of the monodromy acting on $\Gr_{F}^{0}H^{q}_{\lim}$, and
    \begin{displaymath}
        \beta^{p,q}=\sum_{r=-k}^{k}r\dim\Gr_{F}^{p}\Gr^{W}_{k+r}H^{k}_{\lim},\quad\text{with}\quad k=p+q.
    \end{displaymath}
    One can infer directly from \eqref{eq:spectral-genus-Steenbrink-GrF0} and the exact sequence \eqref{Milnorsequence}, that $\alpha^{0,n}=\widetilde{p}_g$, since the monodromy acts trivially on $H^{q}(\Xcal_{0})$ for any $q$. Moreover, $\alpha^{0,q} = 0$ if $q<n$. This latter fact follows from \eqref{Milnorsequence} and \eqref{Milnorsequence1}, paired again with the fact that $H^{q}(\Xcal_0)$ has trivial monodromy for any $q$. Similarly, since by \eqref{Milnorsequence1} the mixed Hodge structure $H^{q}_{\lim}$ is pure for $q<n$, we see that $\beta^{0,q}=0$ for $q<n$. Finally, by \cite[Lemma 4.3]{cdg2}, $\beta^{0,n}=-\beta^{n,0}$, and we have $\beta^{n,0}=\beta$, because $\Gr_{F}^{n}=F^{n}$ and $F^{n}\Gr_{n+r}^{W}H^{n}_{\lim}=0$ if $r<0$. This concludes the proof.

    For the second part of the proposition, we need to show that if $\Xcal_{0}$ has canonical singularities, then $\widetilde{p}_{g}=\beta=0$. We already know that $p_{g}$, and hence $\widetilde{p}_{g}$, vanishes, cf. \textsection\ref{section:genus}. The vanishing of $\beta$ is equivalent to a result of C.-L. Wang \cite[Theorem 2.1 \& Corollary 2.4]{Wang:MRL}, to the effect that the nilpotent operator $N$, associated to the monodromy on $H^{n}_{\lim}$, annihilates $F^{n}H^{n}_{\lim}$. Indeed, if $N$ annihilates $F^{n}H^{n}_{\lim}$, then the image of $F^{n}\Gr^{W}_{n+r}H^{n}_{\lim}$ under $N^{r}$ vanishes for $r\geq 1$, and we know that $N^{r}$ defines an isomorphism $\Gr^{W}_{n+r}H^{n}_{\lim}\to\Gr^{W}_{n-r}H^{n}_{\lim}$. Conversely, if $\beta=0$, then necessarily $F^{n}H^{n}_{\lim}=F^{n}W_{n}H^{n}_{\lim}$. But $N$ sends $F^{n}W_{n}H^{n}_{\lim}$ to $F^{n-1}W_{n-2}H^{n}_{\lim}$ and the latter vanishes, because $F^{n-1}\Gr^{W}_{r}H^{n}_{\lim}=0$ for every $r\leq n-2$, for type reasons.

    We notice that the results in \cite{Wang:MRL} require that $\Xcal_{0}$ be Gorenstein and irreducible. The assumption that $\Xcal_{0}$ is Gorenstein in \emph{op. cit.} is automatic in our case, since $\Xcal$ is smooth. As for the asumption that $\Xcal_{0}$ is irreducible, it is not necessary in our setting. Indeed, since the morphism $\Xcal\to\Delta$ has reduced fibers and $\Xcal$ is smooth, the Stein factorization is of the form $\Xcal\to\Delta^{\prime}\to\Delta$, where $\Delta^{\prime}$ is a disjoint union of discs and $\Delta^{\prime}\to\Delta$ is a trivial covering. Working over the components of $\Delta^{\prime}$ instead of $\Delta$, we reduce to the case that the fibers are connected. Since $\Xcal_{0}$ has canonical singularities by assumption, it is in particular normal, so that connectedness entails irreducibility.  
\end{proof}

\subsection{}\label{subsec:asymptoticholomorphictorsion} Let $\Xcal \to \Delta$ be a degeneration as in \textsection\ref{sec:L2deg}. We conclude by \eqref{sec:Quillendeg} and Proposition \ref{prop:L2deg} the following: 

\begin{corollary}\label{cor:asymptoticholomorphictorsion}
    The analytic torsion of the sheaf of holomorphic functions, endowed with the trivial metric, has the following asymptotic behavior for $t$ close to $0$:
    \begin{displaymath}
        \log \tau(\Xcal_t, \Ocal_{\Xcal_t}) = (-1)^{n}  \left(\frac{\mu}{(n+2)!}- \widetilde{p}_g \right)\log|t|^{2} + (-1)^{n}\beta\log\log|t|^{-1}+O(1),
    \end{displaymath}
where $\beta$ is defined in \eqref{eq:beta}. In particular:
\begin{enumerate}
    \item If the weak form of the conjecture holds, then $\tau(\Xcal_t, \Ocal_{\Xcal_t})^{(-1)^n}$ converges to zero as $t$ goes to 0.  
    \item If the strong form of the conjecture holds, then $\tau(\Xcal_t, \Ocal_{\Xcal_t})^{(-1)^n}=O\left(|t|^{2m/(n+2)!-\varepsilon}\right)$ as $t\to 0$, for every $\varepsilon>0$. Here, $m$ is the number of singular points in $\Xcal_{0}$.
    \item If $\Xcal_{0}$ has canonical singularities, then $\tau(\Xcal_t, \Ocal_{\Xcal_t})^{(-1)^n}=O(|t|^{2\mu/(n+2)!})$  as $t\to 0$.
\end{enumerate}
\end{corollary}
\qed

Notice that, while the analytic torsion depends on the choice of metrics, the asymptotic behavior is metric independent. Also, for families of curves, the volume of the fibers with respect to a Kähler form on $\Xcal$ is constant, and the $L^2$-norm on holomorphic differentials is independent of the choice of metric. Hence, the asymptotic behavior is then valid for any choice of Kähler metric on the total space $\Xcal$. We remark that, for curves, $\tau(\Xcal_t, \Ocal_{\Xcal_t})^{-1}=\det\Delta^{0,1}_{\overline{\partial}}$ coincides with $\det\Delta^{0,0}_{\ov{\partial}}$, which in turn coincides with the determinant of the Laplace--Beltrami operator up to a topological constant.

\subsection{}\label{subsec:globalisation}
 A situation where the corollary applies is that of a degeneration of projective varieties endowed with an embedding $\Xcal \subseteq \PBbb^M \times \Delta$ over $\Delta$, and such that $\Xcal$ is equipped with a smooth Kähler metric $\omega$ on $\Xcal$ whose cohomology class $[\omega_t]$ for $t \in \Delta \setminus \{0 \} $ is given by the canonical polarization coming from $\PBbb^M$. 

A variant of the above arises from the tautological family $\Hcal$ of hypersurfaces of degree $d$ in $\PBbb^{n+1}$ over the space of parameters $\PBbb^N$, where $N = \binom{n+1}{d}-1$. Then $\Hcal$ is a smooth space and inherits a Kähler metric from $\PBbb^{n+1}$ and $\PBbb^N$. Hence, the restriction along any curve $\Delta \to \PBbb^{N}$ in which the total space is still smooth is an example of such a degeneration. 

\subsection{}\label{subsec:reformulation} Since any germ of isolated singularity admits a good compactification as in \textsection\ref{subsec:degeneration}, we readily deduce, from Corollary \ref{cor:asymptoticholomorphictorsion}, a reformulation of our conjecture in terms of analytic torsion.
\begin{proposition}\label{prop:reformulation}
For an isolated singularity $f=0$ of dimension $n\geq 1$, the following are equivalent:
\begin{enumerate}
    \item\label{item:reformulation-1} The weak form, respectively strong form, of the conjecture holds for $f$.
    \item\label{item:reformulation-2} For any good compactification $\Xcal\to\Delta$ of $f$, and any K\"ahler form on $\Xcal$ which is rational on smooth fibers, we have
    \begin{displaymath}
        \begin{split}
            &\tau(\Xcal_t, \Ocal_{\Xcal_t})^{(-1)^{n}}\to 0,\quad\text{as}\quad t\to 0,\\
            &\text{respectively}\quad \tau(\Xcal_t, \Ocal_{\Xcal_t})^{(-1)^{n}}=O(|t|^{2/(n+2)!-\varepsilon}),\quad\text{for every }\quad\varepsilon>0, \quad\text{as}\quad t\to 0.   
        \end{split}
    \end{displaymath}
\end{enumerate}
\end{proposition}
\qed

In the proposition above, one can replace \emph{any good compactification} (resp. \emph{any K\"ahler form}) by \emph{some good compactification} (resp. \emph{some K\"ahler form}). We also bring the reader's attention to the fact that the positivity of $\beta$, established in Proposition \ref{prop:L2deg}, is fundamental to prove that \eqref{item:reformulation-2} implies \eqref{item:reformulation-1}.


\section{Quasi-homogeneous singularities}
In this section, we address several cases of the strong form of the conjecture in the setting of quasi-homogeneous singularities. We also recall some useful notions about Newton polyhedra also utilized in later sections.

\subsection{Newton polyhedra}  Let $f(x_0, \ldots, x_n) = \sum_{k\in \NBbb^{n+1}} a_k x^k$ be a power series with complex coefficients, with $a_{0}=0$ and where we define $x^{k}=x_{0}^{k_{0}}\cdots x_{n}^{k_{n}}$. The Newton diagram of the singularity is constituted of the following polyhedra. The upper Newton polyhedron associated to $f$, denoted by $\Gamma_{+}$, is the convex hull of the set $\bigcup_{a_{k}\neq 0}(k+\RBbb^{n+1}_{+})$. The associated Newton boundary, denoted by $\Gamma$, is the boundary of $\Gamma_{+}$. We denote by $\Gamma_{c}$ the compact Newton boundary, meaning the union of the compact faces of $\Gamma$. The lower Newton polyhedron, denoted by $\Gamma_{-}$, is the union of the lines joining the origin with the points on $\Gamma_{c}$. Since we only deal with lower Newton polyhedra, we will usually refer to these as simply Newton polyhedra. 

Below we display a Newton diagram. In the picture, the lower Newton polyhedron is determined by the vertices $(0,0)$, $A$, $B$, $C$, $D$, and it is delimited by the segments painted in red. The compact Newton boundary has three faces, namely $AB$, $BC$, $CD$. The whole Newton boundary has two more unbounded faces, painted in blue. The upper Newton polyhedron is the region above the Newton boundary. \bigskip

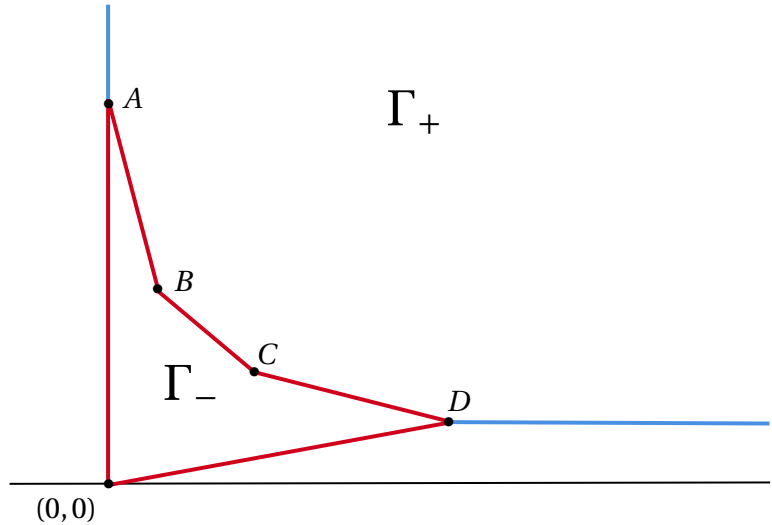
\begin{center}
\tikzset{every picture/.style={line width=0.75pt}} 

\begin{tikzpicture}[x=0.75pt,y=0.75pt,yscale=-0.7,xscale=0.7]

\draw    (48,347.5) -- (596,346.5) ;
\draw [color={rgb, 255:red, 208; green, 2; blue, 27 }  ,draw opacity=1 ][line width=1.5]    (119,70) -- (155.5,209) ;
\draw [color={rgb, 255:red, 208; green, 2; blue, 27 }  ,draw opacity=1 ][line width=1.5]    (155.5,209) -- (224.5,267) ;
\draw [color={rgb, 255:red, 208; green, 2; blue, 27 }  ,draw opacity=1 ][line width=1.5]    (224.5,267) -- (366.5,303) ;
\draw [color={rgb, 255:red, 74; green, 144; blue, 226 }  ,draw opacity=1 ][line width=1.5]    (366.5,303) -- (595.5,304) ;
\draw [color={rgb, 255:red, 208; green, 2; blue, 27 }  ,draw opacity=1 ][line width=1.5]    (121.5,348) -- (366.5,303) ;
\draw [color={rgb, 255:red, 208; green, 2; blue, 27 }  ,draw opacity=1 ][line width=1.5]    (119,70) -- (119,348) ;
\draw [color={rgb, 255:red, 74; green, 144; blue, 226 }  ,draw opacity=1 ][line width=1.5]    (119,2) -- (119,70) ;

\draw (112,342) node [anchor=north west][inner sep=0.75pt]    {$\bullet  $};
\draw (64,355.4) node [anchor=north west][inner sep=0.75pt]    {$(0,0)$};
\draw (112,68) node [anchor=north west][inner sep=0.75pt]    {$\bullet $};
\draw (147.5,201.5) node [anchor=north west][inner sep=0.75pt]    {$\bullet $};
\draw (217,261) node [anchor=north west][inner sep=0.75pt]    {$\bullet $};
\draw (357,297) node [anchor=north west][inner sep=0.75pt]    {$\bullet $};
\draw (126,60) node [anchor=north west][inner sep=0.75pt]    {$A$};
\draw (164,191.4) node [anchor=north west][inner sep=0.75pt]    {$B$};
\draw (224,244) node [anchor=north west][inner sep=0.75pt]    {$C$};
\draw (361,278.4) node [anchor=north west][inner sep=0.75pt]    {$D$};
\draw (156,259.4) node [anchor=north west][inner sep=0.75pt]  [font=\LARGE,color={rgb, 255:red, 0; green, 0; blue, 0 }  ,opacity=1 ]  {$\Gamma _{-}$};
\draw (317,61.4) node [anchor=north west][inner sep=0.75pt]  [font=\LARGE]  {$\Gamma _{+}$};

\end{tikzpicture}
\end{center}
\bigskip

\subsection{}

Let $n\geq 1$ be an integer, and consider a quasi-homogeneous polynomial $f(x_0, \ldots, x_n)$, with rational weights $w_0, \ldots, w_n>0$. This means that for any $\lambda \in \CBbb$, we have $f(\lambda^{w_{0}} x_0, \ldots, \lambda^{w_{n}} x_n) = \lambda \cdot f(x_0, \ldots, x_n)$. We suppose  
\begin{displaymath}
    f(x_0, \ldots, x_n) = 0
\end{displaymath}
has an isolated singularity at the origin. Then, the weights satisfy $w_{i}<1$. 

For quasi-homogeneous singularities, the Milnor number and the spectral genus depend only on the weights, and are given by the following formulas:
\begin{equation}\label{eq:qhommilnor}
    \mu = \prod_{i}\left(\frac{1}{w_{i}}-1\right)
\end{equation}
and
\begin{equation}\label{eq:qhomspectralgenus}
    \widetilde{p}_{g}=\sum \left( 1-k_{0}w_{0}-\cdots-k_{n}w_{n}\right),
\end{equation}
where the sum runs over integers $k_{i}>0$ such that $\sum k_{i}w_{i}<1$. These are the interior lattice points of the Newton polyhedron of the polynomial $f(x_0, \ldots, x_n)$. Eventually, for convenience, we may include the lattice points with $\sum k_{i}w_{i}=1$, since these contribute zero to the sum. For the Milnor number, the formula \eqref{eq:qhommilnor} is given in \cite[Theorem 1]{Milnor-Orlik}. The expression \eqref{eq:qhomspectralgenus} for $\widetilde{p}_{g}$ can be derived from \eqref{eq:spgenusMSaito} and the computation of the spectral polynomial of $f$ (cf. \eqref{eq:spectral-polynomial} for the definition), due to Steenbrink \cite[Example 5.11]{Steenbrink-mixedonvanishing}:
\begin{displaymath}
    \mathrm{Sp}_{f}(T)=\prod_{j}\frac{T^{w_{j}}-T}{1-T^{w_{j}}}.
\end{displaymath}
To this end, it is enough to expand the polynomial into a power series of $T$ with rational exponents, and collect the terms whose exponents are strictly smaller than one. For this purpose, one can ignore the negative $T$ in the numerator, and expand $T^{w_j}/(1-T^{w_j}) = T^{w_j} + T^{2w_j} + \ldots $. One finds that the exponents $\alpha^{\prime}$ strictly less than one are of the form  $\alpha^{\prime}=\sum k_{i}w_{i}$, where the $k_{i}$ run over all the possible integers $k_{i}>0$ such that $\alpha^{\prime}<1$.

\subsection{Homogeneous polynomials} It would be interesting to establish the strong form of the conjecture for quasi-homogeneous singularities, such as the Brieskorn--Pham singularities of the form
\begin{displaymath}
    f(x) = x_0^{a_0} + x_1^{a_1} + \cdots + x_n^{a_n} = 0,
\end{displaymath}
corresponding to the choices of weights $w_i = 1/a_i$. Here we treat the case when $f$ be a homogenenous polynomial in arbitrary dimension, so that all $w_i = \frac{1}{d}$. In this case \eqref{eq:qhommilnor} gives
\begin{equation}\label{eq:hommilnor}
    \mu = (d-1)^{n+1}.
\end{equation}
The spectral genus can be computed explicitly, based on the following elementary identity: 

\begin{lemma}\label{lemma:homogenousspectral}
Let $d > 1$ be an integer, and $n \geq 1$. Then  
    \begin{equation}
    \sum \left(d-k_0-\cdots - k_n\right) = \frac{d(d-1)\cdots (d-(n+1))}{(n+2)!},
\end{equation}
where the sum runs over integers $k_i > 0$, such that $\sum k_i < d$. Consequently, 
\begin{displaymath}
    \widetilde{p}_{g}=\frac{(d-1)\cdots (d-(n+1))}{(n+2)!}.
\end{displaymath}
\end{lemma}

\begin{proof}
    An inductive argument shows that the sum is a polynomial of degree $n+2$ in $d$. Moreover, the sum is empty for $d = 0, \ldots, n+1,$ so  the polynomial is of the form $C \cdot d(d-1)\cdots (d-(n+1))$ for some constant $C$. Since the sum is equal to $1$ for $d=n+2$ one sees that $C = \frac{1}{(n+2)!}$.
\end{proof}

The following proposition follows immediately from \eqref{eq:hommilnor} and Lemma \ref{lemma:homogenousspectral}:
\begin{proposition}\label{prop:conjecture-homogenenous}
Let $f=0$ define an isolated degree $d$ homogeneous singularity at the origin. Then the strong form of the conjecture is true. Moreover, for a fixed $n$, 
    \begin{displaymath}
        \frac{\widetilde{p}_g(d)}{\mu(d)} \nearrow \frac{1}{(n+2)!},\quad  \hbox{ as }\quad d \to +\infty,
    \end{displaymath}
where $\mu(d)$ and $\widetilde{p}_{g}(d)$ denote the corresponding Milnor number and spectral genus.
\end{proposition}
\qed

\subsection{Quasi-homogeneous singularities in dimension one}\label{sec:listofsingularitiesdim1} The purpose of the rest of this section is to prove the following theorem: 
\begin{theorem}\label{thm:quasi-homogeneous}
    When $n =1$, the strong form of the conjecture is satisfied in the case of quasi-homogeneous singularities.
\end{theorem}

In preparation for the proof, we first notice that the invariants $\mu$ and $\widetilde{p}_g$ only depend on the analytical type of the germ of the singularity. By \cite[Satz 1.3]{K-Saito-Inventiones} we can assume that $w_i \leq \frac{1}{2}$. Then, by \cite{Yoshinaga-Suzuki}, the weights depend only on the topological type of the singularity, which is one of the following: 
\begin{equation}\label{eq:qhomdim1:1}
    x^a + y^b = 0, 
\end{equation}
\begin{equation}\label{eq:qhomdim1:2}
    x(x^a + y^b) = 0,
\end{equation}
\begin{equation}\label{eq:qhomdim1:3}
    xy(x^a + y^b) = 0.
\end{equation}
In the following sections, we will prove the theorem by analyzing these cases. 
\subsection{}
We begin our treatment of the singularities in \textsection\ref{sec:listofsingularitiesdim1} by the following computation, which is due to Mordell \cite{Mordell} in the case when $a$ and $b$ are relatively prime:

\begin{proposition}\label{prop:Mordell}
 Let $a, b \geq 2$ be two integers. Define $k 
 = \gcd(a,b)$ and write $a = k \cdot a', b = k \cdot b'$. Then we have 
\begin{equation}\label{eq:mordell-type-sum}
     \sum\left(1-\frac{x}{a}-\frac{y}{b}\right) = \frac{(a-1)(b-1)}{6}- \frac{a^{\prime}+b^{\prime}}{12}(k-1)-\frac{(a^{\prime}-1)(b^{\prime}-1)(a^{\prime}+b^{\prime}+1)}{12a^{\prime}b^{\prime}},
\end{equation}
where the sum runs over the interior lattice points of the triangle with vertices $(0,0)$, $(a,0)$, $(0,b)$.
\end{proposition}
\begin{proof}

The case when $a,b$ are relatively prime follows immediately from the computation of Mordell \cite[Equation (4)]{Mordell}. The sum in \eqref{eq:mordell-type-sum} is
\begin{equation}\label{eq:conj-xa-yb}
    \frac{(a-1)(b-1)}{6}-\frac{(a-1)(b-1)(a+b+1)}{12ab}.
\end{equation}
In general, we describe the main steps of the reasoning and leave the details to the reader. Let $k=\gcd(a,b)$, and decompose $a=a^{\prime}k$ and $b=b^{\prime}k$. We cut the triangle into smaller pieces, as in the  following picture (for which $k=4$):\bigskip

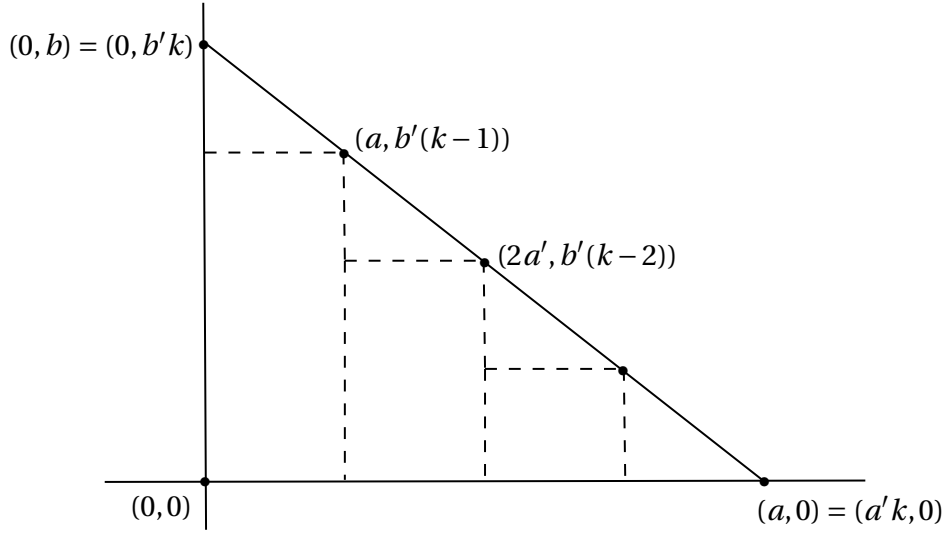
\begin{center}
\tikzset{every picture/.style={line width=0.75pt}} 

\begin{tikzpicture}[x=0.75pt,y=0.75pt,yscale=-0.7,xscale=0.7]

\draw    (119,2) -- (121,382) ;
\draw    (48,347.5) -- (596,346.5) ;
\draw    (119,30) -- (523,347) ;
\draw  [dash pattern={on 4.5pt off 4.5pt}]  (120,110) -- (220,110) ;
\draw  [dash pattern={on 4.5pt off 4.5pt}]  (220,110) -- (221,188) ;
\draw  [dash pattern={on 4.5pt off 4.5pt}]  (221,188) -- (321,188) ;
\draw  [dash pattern={on 4.5pt off 4.5pt}]  (321,188) -- (322,266) ;
\draw  [dash pattern={on 4.5pt off 4.5pt}]  (322,266) -- (422,266) ;
\draw  [dash pattern={on 4.5pt off 4.5pt}]  (422,266) -- (423,344) ;
\draw  [dash pattern={on 4.5pt off 4.5pt}]  (221,188) -- (221,346) ;
\draw  [dash pattern={on 4.5pt off 4.5pt}]  (322,266) -- (322,347) ;

\draw (112,26.5) node [anchor=north west][inner sep=0.75pt]    {$\bullet $};
\draw (516,342) node [anchor=north west][inner sep=0.75pt]    {$\bullet $};
\draw (-24,20.4) node [anchor=north west][inner sep=0.75pt]    {$(0,b)=( 0,b^{\prime}k)$};
\draw (515,355.4) node [anchor=north west][inner sep=0.75pt]    {$(a,0)=( a^{\prime}k,0)$};
\draw (64,355.4) node [anchor=north west][inner sep=0.75pt]    {$(0,0)$};
\draw (213,105) node [anchor=north west][inner sep=0.75pt]    {$\bullet  $};
\draw (314.5,184) node [anchor=north west][inner sep=0.75pt]    {$\bullet  $};
\draw (414,262) node [anchor=north west][inner sep=0.75pt]    {$\bullet  $};
\draw (113,342) node [anchor=north west][inner sep=0.75pt]    {$\bullet  $};
\draw (225,87.4) node [anchor=north west][inner sep=0.75pt]    {$( a,b^{\prime}( k-1))$};
\draw (328,171.4) node [anchor=north west][inner sep=0.75pt]    {$( 2a^{\prime},b^{\prime}( k-2))$};

\end{tikzpicture}
\end{center}
\bigskip
In the picture, there are $k$ small triangles obtained by appropriately translating the triangle of vertices $(0,0)$, $(a^{\prime},0)$, $(0,b^{\prime})$. The rest is divided into rectangles. After taking into account the appropriate translations, the evaluation of the sum in \eqref{eq:mordell-type-sum} restricted to the interior points of the smaller triangles reduces to Mordell's computation \eqref{eq:conj-xa-yb}. The evaluation of the sum \eqref{eq:mordell-type-sum} on the interior lattice points of the rectangles is elementary and reduces to some double sums of consecutive integers. Then we add the contribution from the dashed lines, which are dealt with in the same way. 
This concludes the proof.
\end{proof}
\begin{proposition}
    The strong form of the conjecture is true for the singularities of the form \eqref{eq:qhomdim1:1}.
\end{proposition}
\begin{proof}
Recall by \eqref{eq:qhomspectralgenus} that $\widetilde{p}_g$ is given by the sum in \eqref{eq:mordell-type-sum}. We first assume that $a$ and $b$ are relatively prime. Without loss of generality, we may suppose that $a \geq 2$ and $b \geq 3$. Then, by the formula of Proposition \ref{prop:Mordell} : 
\begin{displaymath}
    \frac{\mu}{6}- \widetilde{p}_g = \frac{(a-1)(b-1)(a+b+1)}{12ab} \geq \frac{1}{6}.
\end{displaymath}
In the general case, the same formula with $k\geq 2$, combined with the value of the Milnor number, gives
\begin{displaymath}
   \frac{\mu}{6}-\widetilde{p}_{g}=\frac{a^{\prime}+b^{\prime}}{12}(k-1)+\frac{(a^{\prime}-1)(b^{\prime}-1)(a^{\prime}+b^{\prime}+1)}{12a^{\prime}b^{\prime}}.
\end{displaymath}
 We then have 
\begin{displaymath}
    \frac{\mu}{6}-\widetilde{p}_{g}\geq \frac{a^{\prime}+b^{\prime}}{12}(k-1)\geq \frac{1}{6}.
\end{displaymath}
This concludes the proof.
\end{proof}

\subsection{} We are now in a position to address the conjecture for the remaining quasi-homogeneous singularities in dimension one, which reduce to Proposition \ref{prop:Mordell}. For these cases, we don't state the explicit expressions for the spectral genera, since they quickly become unmanageable. 

\begin{corollary}
The strong form of the conjecture is true for the singularities of the form \eqref{eq:qhomdim1:2} and \eqref{eq:qhomdim1:3}.
\end{corollary}
\begin{proof}
As in the proof of Proposition \ref{prop:Mordell}, we indicate the main steps of the argument. We begin with the case $x(x^{a}+y^{b})=0$. In this case, $\widetilde{p}_{g}$ is given by the sum

\begin{equation}\label{eq:Mordell-type-sum-xxa-xyb}
    \sum\left(1-\frac{x}{a+1}-\frac{ay}{(a+1)b}\right),
\end{equation}
over the lattice points $(x,y)$ with $x>0$, $y>0$, in the interior of the triangle with vertices $(0,0)$, $(a+1,0)$, $(1,b)$. Actually, such points necessarily lie either in the interior of the triangle $T$ with vertices $(1,0)$, $(a+1,0)$ and $(1,b)$, or on the open edge joining $(1,0)$ and $(1,b)$, as in the following picture:\bigskip

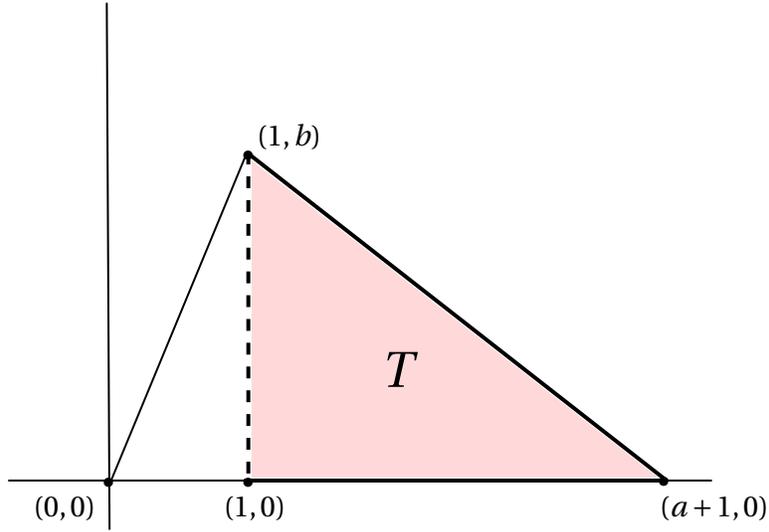
\begin{center}
\tikzset{every picture/.style={line width=0.75pt}} 

\begin{tikzpicture}[x=0.75pt,y=0.75pt,yscale=-0.7,xscale=0.7]

\draw   (119,2) -- (121,382) ;
\draw    (48,346.5) -- (218,346.5) ;
\draw    (518,346.5) -- (555,346.5) ;
\draw[line width=1.5] (218,346.5) -- (518,346.5);
\draw[line width=1.5]   (220,110) -- (523,347) ;
\draw[line width=1.5]     [dash pattern={on 4.5pt off 4.5pt}]  (221,110) -- (221,346) ; 
\draw    (220,110) -- (121.5,348) ;

\draw (213.5,106) node [anchor=north west][inner sep=0.75pt]    {$\bullet $};
\draw (513,341) node [anchor=north west][inner sep=0.75pt]    {$\bullet $};
\draw (515,355.4) node [anchor=north west][inner sep=0.75pt]    {$( a+1,0)$};
\draw (225,87.4) node [anchor=north west][inner sep=0.75pt]    {$( 1,b)$};
\draw (113,342) node [anchor=north west][inner sep=0.75pt]    {$\bullet  $};
\draw (213.5,342) node [anchor=north west][inner sep=0.75pt]    {$\bullet $};
\draw (201,355.4) node [anchor=north west][inner sep=0.75pt]    {$( 1,0)$};
\draw (64,355.4) node [anchor=north west][inner sep=0.75pt]    {$(0,0)$};

\draw[red!15, fill=red!15] (224,343.5) -- (514,343.5) -- (224,117) -- cycle;
\draw (314,251.4) node [anchor=north west][inner sep=0.75pt]  [font=\LARGE]  {$T$};

\end{tikzpicture}
\end{center}
\bigskip
The triangle $T$, colored in pink, can be translated by one unit to the left so that, after the corresponding change of variables, the evaluation of \eqref{eq:Mordell-type-sum-xxa-xyb} on the interior lattice points of $T$ reduces to the case treated in Proposition \ref{prop:Mordell}. The contribution of the lattice points on the open edge between $(1,0)$ and $(1,b)$, represented by the dashed line, reduces to a sum of consecutive integers. These computations are then combined with the value of the Milnor number, now given by $\mu=(a+1)(b-1)+1$. One concludes by inspection of the obtained expressions that $\mu/6-\widetilde{p}_{g}\geq 1/6$. 

Next for the singularity $xy(x^{a}+y^{b})=0$. In this case, the spectral genus is given by the sum
\begin{equation}\label{eq:Mordell-type-sum-xyxa-xyyb}
    \sum\left(1-\frac{bx}{(a+1)(b+1)-1}-\frac{ay}{(a+1)(b+1)-1}\right),
\end{equation}
over the lattice points $(x,y)$ either in the interior of the triangle $T$ with vertices $(1,1)$, $(a+1,1)$ and $(1,b+1)$, or on the edge joining $(1,1)$ and $(a+1,1)$, or on the edge joining $(1,1)$ and $(1,b+1)$. Notice that $(1,1)$ is one such point. This is represented in the following picture, where $T$ is again colored in pink:\bigskip

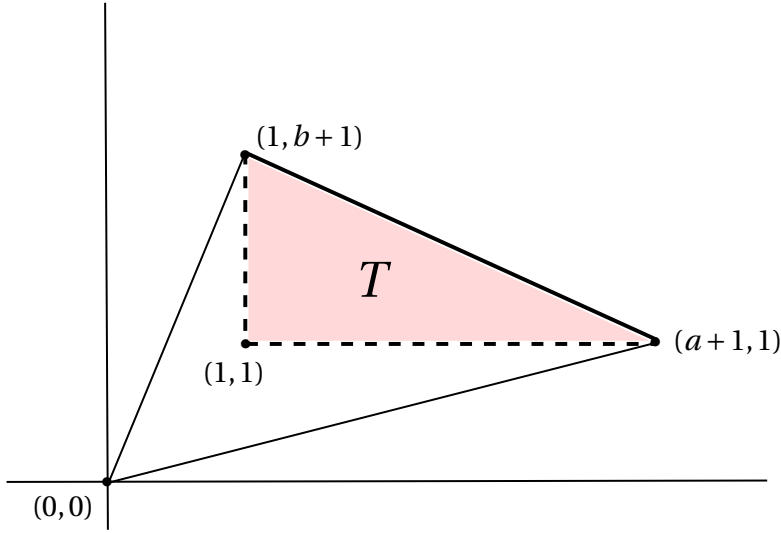
\begin{center}
\tikzset{every picture/.style={line width=0.75pt}} 

\begin{tikzpicture}[x=0.75pt,y=0.75pt,yscale=-0.7,xscale=0.7]

\draw    (119,2) -- (121,382) ;
\draw    (48,347.5) -- (596,346.5) ;
\draw[line width=1.5]    (220,110) -- (516.5,245) ;
\draw[line width=1.5]  [dash pattern={on 4.5pt off 4.5pt}]  (220,110) -- (220,248) ;
\draw    (220,110) -- (121.5,348) ;
\draw    (516.5,247) -- (121.5,348) ;
\draw[line width=1.5]  [dash pattern={on 4.5pt off 4.5pt}]  (220.5,248) -- (516.5,248) ;

\draw (212.5,106) node [anchor=north west][inner sep=0.75pt]    {$\bullet $};
\draw (508,241) node [anchor=north west][inner sep=0.75pt]    {$\bullet $};
\draw (526,235.4) node [anchor=north west][inner sep=0.75pt]    {$( a+1,1)$};
\draw (225,87.4) node [anchor=north west][inner sep=0.75pt]    {$( 1,b+1)$};
\draw (113,342) node [anchor=north west][inner sep=0.75pt]    {$\bullet  $};
\draw (296,186.4) node [anchor=north west][inner sep=0.75pt]  [font=\LARGE]  {$T$};
\draw (64,355.4) node [anchor=north west][inner sep=0.75pt]    {$(0,0)$};
\draw (187,258.4) node [anchor=north west][inner sep=0.75pt]    {$( 1,1)$};
\draw (213,242) node [anchor=north west][inner sep=0.75pt]    {$\bullet $};

\draw[red!15, fill=red!15] (223,115) -- (508,245) -- (223,245) -- cycle;

\draw (296,186.4) node [anchor=north west][inner sep=0.75pt]  [font=\LARGE]  {$T$};
\end{tikzpicture}
\end{center}
\bigskip
The triangle can be translated so that $(1,1)$ is sent to the origin. The evaluation of \eqref{eq:Mordell-type-sum-xyxa-xyyb} on the interior points of $T$ is then covered by Proposition \ref{prop:Mordell}. The sums on lattice points on the edges reduce to sums of consecutive integers. The result of the computation is then combined with the value of the Milnor number $\mu=(a+1)(b+1)$, and an examination of the expression yields again the bound $\mu/6-\widetilde{p}_{g}\geq 1/6$.
\end{proof}

\section{Irreducible curve singularities}
In this section, we consider the conjecture for an irreducible germ of a plane curve singularity, defined by $f\colon (\CBbb^{2},0)\to (\CBbb,0)$. In this case, the spectrum can be described explicitly in terms of Puiseux pairs. This is discussed in an unpublished paper of M. Saito \cite{MSaito:Hertling}, whose presentation and notation we follow. The same result can be derived from the work of Schrauwen--Steenbrink--Stevens \cite{SSS}, as discussed by N\'emethi in \cite{Nemethi:spectrum-curves}. 

\subsection{} After possibly changing variables, the equation $f(x,y)=0$ is equivalent to a Puiseux series representation for $y$ in terms of $x$:
\begin{displaymath}
    \begin{split}
        y=&\sum_{1\leq i\leq [k_{1}/n_{1}]}c_{0,i}x^{i}+\sum_{0\leq i\leq [k_{2}/n_{2}]}c_{1,i}x^{(k_{1}+i)/n_{1}}\\
        &+\sum_{0\leq i\leq [k_{3}/n_{3}]}c_{2,i}x^{k_{1}/n_{1}+(k_{2}+i)/n_{1}n_{2}}
        +\cdots+\sum_{i\geq 0}c_{g,i}x^{k_{1}/n_{1}+k_{2}/n_{1}n_{2}+\cdots+(k_{g}+i)/n_{1}\cdots n_{g}}.
    \end{split}
\end{displaymath}
The pairs $(k_{i},n_{i})$ are called Puiseux paris, and they fulfill the following properties:
\begin{enumerate}
    \item $k_{i}$ and $n_{i}$ are coprime integers.
    \item $n_{i}>1$.
    \item $k_{1}>n_{1}$ and $k_{i}>k_{i-1}n_{i}$, for $i\geq 2$. In particular, $k_{i}>1$ for all $i$. 
\end{enumerate}
In addition, we introduce positive integers $w_{i}$, defined recursively by
\begin{displaymath}
    w_{1}=k_{1},\quad w_{i}=n_{i-1}n_{i}w_{i-1}+k_{i}\quad\text{for}\quad i\geq 2.
\end{displaymath}
Hence, for $i\geq 2$ we can write
\begin{equation}\label{eq:expression-w}
    w_{i}=\left(\sum_{j=1}^{i-1}k_{j}n_{j}(n_{j+1}\cdots n_{i-1})^{2}\right)n_{i}+k_{i}.
\end{equation}
For the sake of clarity, we stress that the term of index $j=i-1$ in the sum is understood to be $k_{i-1}n_{i-1}$. We also observe that, because the integers $k_{i}$ and $n_{i}$ are coprime, the same holds for $w_{i}$ and $n_{i}$. 

\subsection{} We next express the Milnor number and the spectral genus in terms of the previous quantities. 
\begin{lemma}\label{lemma:spectrum-irreducible}
Let $f\colon (\CBbb^{2},0)\to (\CBbb,0)$ define an irreducible germ of a plane curve singularity. Let the notation be as above, and define $n_{i}^{\prime}=n_{i+1}\cdots n_{g}$, with the convention $n_{g}^{\prime}=1$. Then:
\begin{enumerate}
    \item The Milnor number is given by
\begin{displaymath}
    \mu=\sum_{i=1}^{g}(n_{i}-1)(w_{i}-1)n_{i}^{\prime}.
\end{displaymath}
\item The spectral genus is given by
\begin{displaymath}
    \widetilde{p}_{g}=\sum_{i=1}^{g}\sum\left(1-\frac{1}{n_{i}^{\prime}}\left(k+\frac{x}{n_{i}}+\frac{y}{w_{i}}\right)\right),
\end{displaymath}
where the second sum runs over the integers $k\geq 0$, $x>0$ and $y>0$, satisfying $k < n_{i}^{\prime}$ and $x/n_{i}+y/w_{i}<1$. 
\end{enumerate}
\end{lemma}
\begin{proof}
See N\'emethi \cite[Section 3]{Nemethi:spectrum-curves}, and in particular Theorem 3.1 therein, and M. Saito \cite[Theorem 1.5 \& Section 5]{MSaito:Hertling}. For the spectral genus, we refer to \textsection\ref{subsec:conventions-spectrum} above for the expression in terms of M. Saito's convention. 
\end{proof}

\begin{corollary}
With the assumptions and notation as above, we have
\begin{displaymath}
    \frac{\mu}{6}-\widetilde{p}_{g}=\frac{1}{12}\sum_{i=1}^{g}S_{i},
\end{displaymath}
where $S_{i}=S_{i}^{+}-S_{i}^{-}$ and 
\begin{displaymath}
    S_{i}^{+}=\frac{(n_{i}-1)(w_{i}-1)(n_{i}+w_{i}+1)}{n_{i}w_{i}},\quad S_{i}^{-}=(n_{i}-1)(w_{i}-1)(n_{i}^{\prime}-1).
\end{displaymath}
\end{corollary}
\begin{proof}
This is an elementary computation using the expressions provided in Lemma \ref{lemma:spectrum-irreducible} and applying Proposition \ref{prop:Mordell} in the case when $a$ and $b$ are coprime.
\end{proof}

\subsection{} We are now in a position to state and prove the main theorem of this section.
\begin{theorem}\label{thm:irreducible-curve}
The strong form of the conjecture holds for germs of irreducible plane curve singularities. More precisely,
\begin{displaymath}
    \frac{\mu}{6}-\widetilde{p}_{g}\geq \frac{1}{12}S_{1}^{+}\geq\frac{1}{6},
\end{displaymath}
and the first inequality is strict if there are two or more Puiseux pairs. 
\end{theorem}

\begin{proof}
Observe that $S_{g}^{-}=0$, since $n_{g}^{\prime}=1$. Hence, the case $g=1$ is trivial and we may assume that $g\geq 2$. We have to prove
\begin{displaymath}
    \sum_{i=2}^{g}S_{i}^{+}-\sum_{i=1}^{g-1}S_{i}^{-}> 0.
\end{displaymath}

First of all, for $i\geq 2$, we provide a lower bound for $S_{i}^{+}$ which is simpler to deal with. For this, using that $k_{i}>1$ in \eqref{eq:expression-w}, and that $(n_{i}+w_{i}+1)/w_{i}>1$, we see that
\begin{equation}\label{eq:Si++}
    S_{i}^{+}>T_{i}^{+}:=(n_{i}-1)\sum_{j=1}^{i-1}k_{j}n_{j}(n_{j+1}\cdots n_{i-1})^{2},
\end{equation}
where the summand of index $j=i-1$ is understood to be $k_{i-1}n_{i-1}$. 

Next, for all $i$, $S_{i}^{-}$, has the following simple upper bound:
\begin{equation}\label{eq:Si-}
    \begin{split}
        S_{i}^{-}\leq T_{i}^{-}:&=(n_{i}-1)(n_{i}^{\prime}-1)w_{i}\\
        &=(n_{i}-1)(n_{i+1}\cdots n_{g}-1)\left(k_{i}+\sum_{j=1}^{i-1}k_{j}n_{j}(n_{j+1}\cdots n_{i-1})^{2}n_{i}\right).
    \end{split}
\end{equation}

Using the above bounds for $S_{i}^{+}$ and $S_{i}^{-}$, we can write
\begin{equation}\label{eq:decomposition-coefficients}
     \sum_{i=2}^{g}S_{i}^{+}-\sum_{i=1}^{g-1}S_{i}^{-}>\sum_{i=2}^{g}T_{i}^{+}-\sum_{i=1}^{g-1}T_{i}^{-}=a_{1}k_{1}+\cdots+a_{g}k_{g},
\end{equation}
and it is enough to prove that the coefficients $a_{i}$ are positive. We will discuss the coefficient $a_{1}$. The other coefficients are dealt with similarly.

We denote by $P_{i}$ the coefficient of $k_{1}$ in $T_{i}^{+}$. Similarly, we denote by $N_{i}$ the coefficient of $k_{1}$ in $T_{i}^{-}$. Inspecting the expanded expressions \eqref{eq:Si++} and \eqref{eq:Si-} for $T_{i}^{+}$ and $T_{i}^{-}$, we find the following. For the coefficients $P_{i}$, we have
\begin{displaymath}
    P_{2}=n_{1}(n_{2}-1)
\end{displaymath}
and
\begin{displaymath}
    P_{i}=n_{1}(n_{2}\cdots n_{i-1})^{2}(n_{i}-1)
\end{displaymath}
 for $i\geq 3$. For the coefficients $N_{i}$, we have
\begin{displaymath}
    N_{1}=(n_{1}-1)(n_{2}\cdots n_{g}-1)
\end{displaymath}
and
\begin{displaymath}
    N_{2}=n_{1}n_{2}(n_{2}-1)(n_{3}\cdots n_{g}-1),
\end{displaymath}
while for $i\geq 3$ we decompose $N_{i}=N_{i,a}+N_{i,b}$, where
\begin{displaymath}
    N_{i,a}=n_{1}(n_{2}\cdots n_{i})^{2}n_{i+1}\cdots n_{g}-n_{1}(n_{2}\cdots n_{i-1})^{2}n_{i}\cdots n_{g}
\end{displaymath}
and
\begin{displaymath}
    N_{i,b}=n_{1}(n_{2}\cdots n_{i-1})^{2}n_{i}-n_{1}(n_{2}\cdots n_{i})^{2}.
\end{displaymath}

The sum of the coefficients $N_{i,a}$ is a telescopic sum, with value
\begin{displaymath}
    \sum_{i=3}^{g-1}N_{i,a}=n_{1}(n_{2}\cdots n_{g-1})^{2}n_{g}-n_{1}n_{2}^{2}n_{3}\cdots n_{g}. 
\end{displaymath}
Adding the contributions of $P_{2}$, $N_{1}$ and $N_{2}$, we obtain
\begin{equation}\label{eq:Nia-telescopic-P2-N1-N2}
    P_{2}-N_{1}-N_{2}-\sum_{i=3}^{g-1}N_{i,a}=n_{2}\cdots n_{g}+n_{1}n_{2}^{2}-n_{1}(n_{2}\cdots n_{g-1})^{2}n_{g}-1. 
\end{equation}

Next, we consider the coefficients $P_{i}$ together with the coefficients $N_{i,b}$, for $3\leq i\leq g-1$. More precisely,
\begin{displaymath}
    P_{i}-N_{i,b}=n_{1}(n_{2}\cdots n_{i})^{2}-n_{1}(n_{2}\cdots n_{i-1})^{2},
\end{displaymath}
which again gives rise to a telescopic sum:
\begin{displaymath}
    \sum_{i=3}^{g-1}(P_{i}-N_{i,b})=n_{1}(n_{2}\cdots n_{g-1})^{2}-n_{1}n_{2}^{2}.
\end{displaymath}
Adding the coefficient $P_{g}=n_{1}(n_{2}\cdots n_{g-1})^{2}(n_{g}-1)$, we find
\begin{equation}\label{eq:Pi-Nib-telescopic-Pg}
    \sum_{i=3}^{g-1}(P_{i}-N_{i,b}) + P_{g}=n_{1}(n_{2}\cdots n_{g-1})^{2}n_{g}-n_{1}n_{2}^{2}.
\end{equation}

To conclude, we add \eqref{eq:Nia-telescopic-P2-N1-N2} and \eqref{eq:Pi-Nib-telescopic-Pg}, which yields the coefficient $a_{1}$ in \eqref{eq:decomposition-coefficients}:
\begin{displaymath}
    a_{1}=n_{2}\cdots n_{g}-1>0.
\end{displaymath}
This completes the proof.
\end{proof}

\section{Generic singularities with large Newton polyhedra}\label{section:large-Newton}

In this section, we consider isolated hypersurface singularities which are non-degenerate and convenient in the sense of Kouchnirenko \cite{Kouchnirenko}.\footnote{In \emph{op. cit.}, which is written in French, convenient is called \emph{commode}.} These conditions are formulated in terms of the Newton polyhedron. The non-degeneracy condition is shown to be generic for the Zariski topology. Therefore, we may refer to such singularities as generic. Convenient means that the Newton polyhedron intersects all the coordinate hyperplanes. We refer to \emph{op. cit.} for details. 

\subsection{}\label{subsec:Kouchnirenko} Consider now a convenient and generic isolated hypersurface singularity defined by $f\colon (\CBbb^{n+1},0)\to (\CBbb,0)$. The formula of Kouchnirenko \cite[Th\'eor\`eme I]{Kouchnirenko} states that 
\begin{equation}\label{eq:Kouch}
    \mu = (n+1)!\vol(\Gamma_{-})-n! \vol_{n}(\Gamma_{-})+\ldots+(-1)^{n}\vol_{1}(\Gamma_{-})+(-1)^{n+1}.
\end{equation}
Here, $\vol$ is the standard volume in $\RBbb^{n+1}$, and $\vol_{k}(\Gamma_{-})$ refers to the volume of the lower Newton polyhedron intersected with all the coordinate subspaces of dimension $k$, with respect to the standard Lebesgue measure. In particular, $\vol_{n}(\Gamma_{-})$ is the volume of $\Gamma_{-}$ intersected with the coordinate hyperplanes.

The spectral genus can also be described explicitly in terms of the Newton polyhedron. Let $\Gamma_{c}$ be the compact Newton boundary. It is defined by a homogeneous, concave, piece-wise linear function $\phi$, which takes the value 1 on the faces of $\Gamma_{c}$. The lower Newton polyhedron can then be presented as $\Gamma_{-}=\lbrace x\in\RBbb_{+}^{n+1}\mid\phi(x)\leq 1\rbrace$. By \cite{MSaito-Newton}, in particular using the formulas on the first page, it follows that the spectral genus is given by
\begin{equation}\label{eq:spectral-genus-saito}
    \widetilde{p}_{g}=\sum_{x\in\Gamma_{-}^{\circ}\cap\ZBbb^{n+1}} (1-\phi(x)),
\end{equation}
where the sum runs over the interior lattice points of $\Gamma_{-}$. Here we have taken into account the different normalization as recalled in  \eqref{eq:spgenusMSaito}. This sum is analogous to \eqref{eq:qhomspectralgenus} but appears here in an a priori different context.
\subsection{} In order to estimate an asymptotic form of the sum \eqref{eq:spectral-genus-saito}, we will apply a rough version of the Euler--Maclaurin formula for polyhedra. For the statement, we adopt the following convention for Lebesgue measures. Let $P$ be a convex lattice polyhedron in $\RBbb^{d}$, and $Q$ a face of it. Let $\langle Q\rangle_{\RBbb}\subseteq\RBbb^{d}$ be the real vector space parallel to $Q$. Then we denote by $dx$ the Lebesgue measure on $\langle Q\rangle_{\RBbb}$ which gives volume one to the fundamental domain of the lattice $\ZBbb^{d}\cap\langle Q\rangle_{\RBbb}$.
\begin{lemma}[Berline--Vergne]\label{lemma:BerlineVergne}
 Let $P\subset\RBbb^{d}$ be a compact convex lattice polyhedron, and let $h$ be a smooth function on $P$. Then, there exists an asymptotic expansion
 \begin{equation}\label{eq:EM}
    \sum_{x\in (kP)\cap\ZBbb^{d}} h(x/k)= k^{d}\int_{P}h(x)dx +\frac{k^{d-1}}{2}\int_{\partial P}h(x)dx+O(k^{d-2}),\quad\text{as}\quad k\to +\infty.
 \end{equation}
\end{lemma}
\begin{proof}
This is an application of the local Euler--Maclaurin formula of Berline--Vergne \cite[Section 5.4, Theorem 5]{BerlinVergnelocalasym}. The dominant term of the expansion is already given in \emph{loc. cit.} They discuss the subdominant term in Section 5.5. Notice also that in \cite{BerlinVergnelocalasym} one supposes that $h$ is a smooth function on $\RBbb^{d}$ with compact support. Since the sum and the first terms of the expansion depend only on the values of $h$ in $P$, which is compact, we can smoothly extend $h$ outside of $P$, still with compact support, for which the formula applies. 
\end{proof}

\subsection{} In general, the integrals appearing in the previous lemma can be evaluated by applying results of Brion \cite[Section 3.2]{Brion}, see also \cite[Corollary 6.1.10]{BPS}. In our setting, it will be enough to have the following.


\begin{lemma}\label{lemma:centermass}
Let $P$ be a $d$-dimensional simplex in $\RBbb^{d}$, with vertices $0, u_{1},\ldots, u_{d}$. Let $\lambda\in\RBbb^{d}$ be such that $\langle u_{i},\lambda\rangle=1$ for all $i$. Then 
\begin{displaymath}
    \int_{P} \left(1 - \langle x,\lambda\rangle 
    \right) dx=\frac{1}{d+1}\vol(P).
\end{displaymath}
\end{lemma}
\begin{proof}  The formula is equivalent to the statement that \begin{displaymath}
    \int_{P} \langle x,\lambda\rangle 
     dx=\frac{d}{d+1}\vol(P).
\end{displaymath} 
Since both sides change the same way with respect to a linear change of variables, we can suppose that the $u_i$ are the standard basis of $\RBbb^{d}$ and $\lambda = (1,\ldots, 1)$. In this case, the integral amounts to a sum of $d$ terms of the form $\int_P x_j dx$. By a standard computation, 
\begin{displaymath}
    \frac{1}{\vol(P)}\int_P x_j dx = \hbox{ center of mass in the direction } x_j =  \frac{1}{d+1},
\end{displaymath}
which allows to conclude.
\end{proof}

\begin{theorem}\label{thm:asymptotic}
    Let $f(x_0, \ldots, x_n)=0$ be a convenient and generic isolated singularity. Then, for $k$ sufficiently large, the isolated singularity $f(x_0^k, \ldots, x_n^k) = 0$ satisfies the strong form of the conjecture. 
\end{theorem}

\begin{proof}

First of all, notice that the singularity $f(x_0^k, \ldots, x_n^k) = 0$ is still generic and convenient. Hence, the discussion \textsection\ref{subsec:Kouchnirenko} applies to it. If we denote the corresponding Milnor number by $\mu(k)$, then by \eqref{eq:Kouch} we have the asymptotic expansion 
\begin{equation}\label{eq:muk}
    \frac{\mu(k)}{(n+2)!}  = \frac{k^{n+1}}{n+2} \vol(\Gamma_{-}) - \frac{k^{n}}{(n+1)(n+2)} \vol_{n}(\Gamma_{-}) + O(k^{n-1}).
\end{equation}
Similarly, we consider 
\begin{equation}\label{eq:pgk}
    \widetilde{p}_g(k) = \sum_{x \in (k\Gamma_{-}^{\circ}) \cap \ZBbb^{n+1}} \left(1- \phi(x/k)  \right),
\end{equation}
where $\phi$ is the piece-wise linear function determining the lower Newton polyhedron $\Gamma_{-}$. To estimate \eqref{eq:pgk}, we apply Lemma \ref{lemma:BerlineVergne}. For this, we first decompose $\Gamma_{-}$ into simplices with 0 as a vertex and the rest of the vertices in $\ZBbb^{n+1}\cap\Gamma_{c}$. These define a simplicial complex, whose top dimensional simplices we denote by $\sigma_{i}$, and whose codimension 1 simplices we denote by $\tau_{j}$. On each simplex, the function $\phi$ is linear. A first application of the lemma gives
\begin{displaymath}
    \widetilde{p}_{g}(k) = \sum_i \sum_{x \in k \sigma_i \cap \ZBbb^{n+1}} (1-\phi(x/k)) - \sum_j \sum_{x \in k \tau_j \cap \ZBbb^{n+1}} (1-\phi(x/k))+ O(k^{n-1}).
\end{displaymath}
Applying Lemma \ref{lemma:BerlineVergne} once again, together with Lemma \ref{lemma:centermass}, one straightforwardly deduces that 
\begin{equation}\label{eq:pgkbis}
    \widetilde{p}_{g}(k)= \frac{k^{n+1}}{n+2} \vol(\Gamma_{-}) - \frac{k^n}{2(n+1) } \vol_n(\Gamma_{-}) + O(k^{n-1}).
\end{equation}
Combining \eqref{eq:muk} and \eqref{eq:pgkbis} we find that 
\begin{displaymath}
    \frac{\mu(k)}{(n+2)!} - \widetilde{p}_g(k) = k^n \frac{n}{2 (n+1)(n+2)  }\vol_n(\Gamma_{-})+O(k^{n-1}).
\end{displaymath}
For big enough $k$, the leading term, which is positive, dominates. This concludes the proof of the theorem.
\end{proof}

\begin{corollary}\label{cor:asymptotic}
With the assumptions of Theorem \ref{thm:asymptotic}, we have
\begin{displaymath}
    \frac{\widetilde{p}_{g}(k)}{\mu(k)}\nearrow \frac{1}{(n+2)!}\quad\text{ as }\quad k \to +\infty,
\end{displaymath}
where $\mu(k)$ and $\widetilde{p}_{g}(k)$ are the Milnor number and spectral genus of $f(x_0^k, \ldots, x_n^k) = 0$.
\end{corollary}
\qed

\section{Relation to other conjectures}
In this section, we discuss the relationship between our conjecture and other conjectural statements. We begin by showing that the weak conjecture can be derived from the Durfee-type conjecture. Then, we consider the relationship with the conjectures of K. Saito and C. Hertling, which focus on the distribution of the spectral numbers. 

\subsection{Durfee-type conjecture}
Recall from the introduction the Durfee-type conjecture \eqref{eq:Durfee-type}. In order to relate it to the weak conjecture, we begin by expressing the spectral genus in terms of the geometric genus of a suspension. 

Let $f=0$ be an isolated hypersurface singularity of dimension $n\geq 1$, and recall that we denote by $T_{s}$ the semi-simple part of the monodromy acting on $\Gr_{F}^{n}H^{n}(\Mil_{f})$. Let $k\geq 1$ be an integer such that $T_{s}^{k}=1$. We define the $(k+1)$-suspension of $f$ as 
\begin{displaymath}
    h(x_{0},\ldots, x_{n+1})=f(x_{0},\ldots, x_{n})+x_{n+1}^{k+1}.
\end{displaymath}
This function defines an isolated hypersurface singularity at $0$, of dimension $n+1$. 

\begin{lemma}\label{lemma:burgos}
With the notation as above, the spectral genus $\widetilde{p}_{g,f}$ of $f$ and the geometric genus $p_{g,h}$ of $h$ are related by
\begin{displaymath}
    p_{g,h}=k\cdot\widetilde{p}_{g,f}.
\end{displaymath}
\end{lemma}
\begin{proof}
Let $\alpha_{i}^{\prime}$ be the spectral numbers of $f$, taken in $(0,n+1)$. Recall from \textsection\ref{subsec:various-spectral-preliminaries} the Thom--Sebastiani property for the spectral numbers. The spectral numbers of $h$ are thus given by
\begin{displaymath}
    \beta_{i,j}=\alpha_{i}^{\prime}+\frac{j}{k+1},\quad\text{for}\ i=1,\ldots,\mu_{f}\ \text{and}\ j=1,\ldots,k.
\end{displaymath}
The geometric genus of $h$ equals $\#\lbrace \beta_{i,j}\leq 1\rbrace$, by the very definition of the spectral numbers, cf. \textsection\ref{subsec:conventions-spectrum}, and because $p_{g,h}=\dim H^{n+1}(\Mil_{h})$, cf. \textsection\ref{section:genus}. 

Define $\lambda_{i}=1-\alpha_{i}^{\prime}$. The condition $\beta_{i,j}\leq 1$ entails $\alpha_{i}^{\prime}<1$, and is actually equivalent to
\begin{equation}\label{eq:burgos}
    j\leq k\lambda_{i}+\lambda_{i}.
\end{equation}
Because $T_{s}^{k}=1$, the quantities $k\lambda_{i}$ are integers. Moreover, $\lambda_{i}<1$. Therefore, given any $\alpha_{i}^{\prime}<1$, the equation \eqref{eq:burgos} has exactly $k\lambda_{i}$ solutions. Recalling the expression \eqref{eq:spgenusMSaito} for $\widetilde{p}_{g,f}$, we find
\begin{displaymath}
    p_{g,h}=\#\lbrace \beta_{i,j}\leq 1\rbrace=\sum_{\alpha_{i}^{\prime}<1}k(1-\alpha_{i}^{\prime})=k\cdot\widetilde{p}_{g,f}.
\end{displaymath}
This concludes the proof.
\end{proof}

\begin{proposition}\label{prop:burgos}
With the notation as above, the weak conjecture for $f$ is equivalent to the Durfee-type conjecture for $h$.
\end{proposition}
\begin{proof}
The Milnor number of $h$ is $\mu_{h}=k\cdot\mu_{f}$. Hence, by the previous lemma, the Durfee-type conjecture for $h$ amounts to
\begin{displaymath}
    k\cdot\widetilde{p}_{g,f} < \frac{k\cdot\mu_{f}}{(n+2)!} ,
\end{displaymath}
which is equivalent to the weak conjecture for $f$.
\end{proof}

\subsection{}\label{subsec:proof-thmB} We now apply the previous proposition to address Theorem B in the introduction.\bigskip

\noindent\emph{Proof of Theorem B}. The suspension of a quasi-homogeneous singularity is again a quasi-homogeneous singularity. For these, the Durfee-type conjecture is a result of Yau--Zhang \cite[Theorem 1.2]{Yau-Zhang}. We conclude by Proposition \ref{prop:burgos}.

For the case of isolated plane curve singularities, N\'emethi studies in \cite{Nemethi-selecta-2} the Durfee conjecture for suspensions of the form $h(x,y,z)=f(x,y)+z^{N}$. He requires that $N\geq 2$ is coprime with the multiplicities of the exceptional divisors of the minimal embedded resolution of singularities of $f=0$. We can apply his results with $N=k+1$, where $k$ is any multiple of the least common multiple of the multiplicities of the exceptional divisors. Then $T_{s}^{k}=1$ by \cite[Th\'eor\`eme 3]{ACampo}, so that we are in the setting of Proposition \ref{prop:burgos}. In \cite[Theorem 5.1]{Nemethi-selecta-2}, it is proven that
\begin{displaymath}
    p_{g,h} \leq  \frac{\mu_{h}}{6}.
\end{displaymath}
We will check that, apart possibly from some exceptional cases, this is actually a strict inequality. The exceptional cases will be dealt with separately. 

Let $\sigma$ be the signature of the intersection form on $H^{2}(\Mil_{h},\ZBbb)$. N\'emethi observes that, for $N$ chosen as above, we have $\sigma+\mu_{h}=4p_{g,h}$. We thus need to show that $\sigma<-\mu_{h}/3$. If $f$ is not equivalent to a singularity of the form $A_{n}$, $A_{2n,2m}$, $D_{2n+3}$ or $E_{6}$, then, recalling that $\mu_{h}=(N-1)\mu_{f}$, the equation in the statement of \cite[Theorem 5.1 (d)]{Nemethi-selecta-2} gives
\begin{displaymath}
    \sigma\leq -\frac{N^{2}-1}{3N}\mu_{f}=-\frac{N+1}{3N}\mu_{h}<-\frac{\mu_{h}}{3}.
\end{displaymath}
Indeed, under the assumption on $f$, the quantity $\epsilon_{f}$ in that statement is 0 by definition. 

If $f$ is equivalent to $A_{n}$, $D_{2n+3}$ or $E_{6}$, then the strong form of the conjecture holds by Theorem \ref{thm:quasi-homogeneous}, since these are all quasi-homogeneous singularities. 

For a singularity of type $A_{2n,2m}$, N\'emethi's treatment shows that $\sigma<-\mu_{h}/3$ if $n+m\geq 5$, but it fails to provide a strict inequality otherwise. Instead, we deliver a simpler argument with a better outcome: the strong conjecture holds as well. Recall first that this singularity has equation $$f(x,y)=(x^{2}+y^{2n+1})(y^{2}+x^{2m+1}).$$ The lower Newton polygon has vertices at the points $(0,0)$, $A=(0,2n+1)$, $B=(2,2)$ and $C=(2m+1,0)$. It is thus convenient, and the compact Newton boundary has two edges, namely the segments $AB$ and $BC$. It is straightforward to see that the principal parts of $f$ corresponding to $AB$ and $BC$ have non-vanishing gradient over $(\CBbb^{\times})^{2}$, so that $f$ is non-degenerate in the sense of Kouchnirenko. Consequently, the discussion in \textsection\ref{subsec:Kouchnirenko} applies, and the spectral genus is given by the formula \eqref{eq:spectral-genus-saito}. The evaluation of the latter is elementary, and we omit the details. The result is
\begin{displaymath}
    \widetilde{p}_{g,f}=\frac{(n+2)^{2}}{2(2n+3)}+\frac{(m+2)^{2}}{2(2m+3)}-\frac{1}{2}.
\end{displaymath}
As for the Milnor number, Kouchnirenko's formula \eqref{eq:Kouch} yields $\mu_{f}=2n+2m+7$. With these expressions at hand, we find 
\begin{displaymath}
    \frac{\mu_{f}}{6}-\widetilde{p}_{g,f}\geq \frac{8}{15}>\frac{1}{6}.
\end{displaymath}
This concludes the proof.

\qed











\subsection{K. Saito's conjecture}
In \cite{KSaito:distribution}, K. Saito proposed a conjectural density describing the distribution of the spectral numbers of families of isolated hypersurfaces singularities, as the singularity gets worse. In such situations, K. Saito's conjecture entails in particular that the inequality in our conjecture is asymptotically an equality, and hence our conjecture is sharp.

\subsection{} A family $\Fcal$ of $n$-dimensional isolated hypersurface singularities, defined by germs of holomorphic functions $f\colon (\CBbb^{n+1},0)\to (\CBbb,0)$, will be called \emph{degenerating} if it is endowed with a filter and the function $\mu\colon\Fcal\to\RBbb$, induced by the Milnor number, tends to infinity. That is,
\begin{displaymath}
    \lim_{\Fcal}\mu=+\infty.
\end{displaymath}
The filter is a technical device which allows us to talk about limits in a rigorous fashion. We informally see $\Fcal$ as a family of singularities whose Milnor numbers converge to infinity.

\subsection{} To state K. Saito's conjecture, let  $f$ be a germ of a holomorphic function defining an isolated hypersurface singularity at the origin. Recall the notion of the associated spectral numbers, that we now take in the form $0<\alpha_{1}^{\prime}\leq\ldots\leq\alpha_{\mu}^{\prime}<n+1$. We define the corresponding spectral probability measure as
\begin{equation}\label{eq:KSaito:conjecture}
    \delta_{f}=\frac{1}{\mu}\sum_{j}\delta_{\alpha_{j}^{\prime}}.
\end{equation}
For a fixed $n$, consider the measure $N(s) ds$ on $[0, n+1]$ for which 
\begin{displaymath}
    \int f(s) N(s) ds = \int_{0 \leq \sum x_i \leq n+1} f\left(\sum x_i\right) dx_0 \cdots dx_n, 
\end{displaymath}
where the $x_i \in [0,1]$. Notice this is a push-forward measure. Indeed, denote by  $\nu$ the measure on $[0,1]^{n+1}$ induced by the Lebesgue measure on $\RBbb^{n+1}$, and introduce the addition function $\Sigma(x_{0},\ldots, x_{n})=\sum x_{i}: [0,1]^{n+1} \to [0, n+1]$.  Then $N(s)ds = \Sigma_{\ast } (\nu)$.

\begin{conjecture}[K. Saito]
For suitable degenerating families of singularities  $\Fcal$, we have the convergence 
\begin{equation}\label{eq:Saito-limit}
    \lim_{\Fcal}\delta_{f}(s)= N(s) ds,
\end{equation}
of probability measures on $\RBbb$, in the strong sense.
\end{conjecture}
It is part of the conjecture to find meaningful families of singularities for which the statement holds. We recall that for probability measures on $\RBbb$, strong convergence means convergence of the measures of any measurable set. This is equivalent to the convergence of the integrals of any bounded measurable function. The convergence in the strong sense is important in the applications below, where we need to integrate functions which are not continuous, but at least measurable. 

\subsection{}\label{subsection:Saito-known-cases} The Fourier transform of the Dirac distribution $\delta_{f}$ can be expressed in terms of the spectral polynomial of $f$ \eqref{eq:spectral-polynomial}. By considering cases when the latter is known, one can prove K. Saito's conjecture in the following situations:
\begin{enumerate}
    \item Quasi-homogeneous singularities, with weights $w_{0},\ldots, w_{n}$ converging to 0, proven by Saito \cite[\textsection 3.7, Example 1]{KSaito:distribution}.
    \item Irreducible plane curve singularities, with Puiseux pairs $(k_{1},n_{1}),\ldots, (k_{g},n_{g})$ and $n_{g}\to +\infty$, proven by Saito \cite[\textsection 3.9, Example 3]{KSaito:distribution}. More generally, Alberich-Carrami\~{n}ana, \`Alvarez Montaner and G\'omez-L\'opez informed us that they have characterized the sequences of irreducible plane curve singularities which satisfy K. Saito's conjecture. We refer to \cite[Theorem 4.0.1]{Alberich} for the precise statement.
    \item For convenient, generic singularities, with large Newton polyhedra as in Section \ref{section:large-Newton} above, proven by Almir\'on--Schulze \cite[Theorem 1.1]{Almiron-Schulze}.
\end{enumerate}

\subsection{} By applying \eqref{eq:Saito-limit} to well-chosen functions, one can derive necessary conditions for a family $\Fcal$ to satisfy K. Saito's conjecture. We state some of these special features.
\begin{proposition}\label{prop:necessary-conditions-Saito}
Let $\Fcal$ be a degenerating family for which K. Saito's conjecture holds. Then:
\begin{enumerate}
    \item\label{item:necessary-conditions-Saito-1} The Hodge numbers of the Milnor fibers satisfy
    \begin{displaymath}
        \lim_{\Fcal}\dim\Gr_{F}^{p}H^{n}(\Mil_{f})=+\infty,\quad\text{for}\quad 0\leq p\leq n.
    \end{displaymath}
    \item\label{item:necessary-conditions-Saito-2} The minimal spectral value $\alpha_{1}^{\prime}$ satisfies
    \begin{displaymath}
        \lim_{\Fcal}\alpha_{1}^{\prime}=0.
    \end{displaymath}
\end{enumerate}
\end{proposition}
\begin{proof}
For the first point, we have
\begin{equation}\label{eq:Saito-limit-Hodge}
    \lim_{\Fcal}\frac{\dim\Gr_{F}^{p}H^{n}(\Mil_{f})}{\mu}=\lim_{\Fcal}\int_{n-p}^{n-p+1}\delta_{f}=\int_{A_{p}}dx_{0}\cdots dx_{n},
\end{equation}
where $A_{p}$ is the region in $[0,1]^{n+1}$ defined by $n-p\leq\sum x_{i}\leq n-p+1$. Because the integral over $A_{p}$ is strictly positive and  $\mu$ converges to $+\infty$ along $\Fcal$, the same must hold for $\dim\Gr_{F}^{p}H^{n}(\Mil_{f})$.

For the second point, fix $\varepsilon>0$. We have
\begin{displaymath}
    \lim_{\Fcal}\frac{\#\lbrace \alpha_{j}^{\prime}<\varepsilon\rbrace}{\mu}=\lim_{\Fcal}\int_{0}^{\varepsilon}\delta_{f}=\int_{\varepsilon A_{n}}dx_{0}\cdots dx_{n}=\frac{\varepsilon^{n+1}}{(n+1)!},
\end{displaymath}
where we used that $A_{n}$ is the standard simplex in $\RBbb^{n+1}_{+}$, whose volume is $1/(n+1)!$. That is, for every $\varepsilon^{\prime}>0$, we can find an element of the filter $\Fcal_{\varepsilon^{\prime}}$ such that, for all $f\in\Fcal_{\varepsilon^{\prime}}$,
\begin{displaymath}
   \left|\frac{\#\lbrace \alpha_{j}^{\prime}<\varepsilon\rbrace}{\mu} - \frac{\varepsilon^{n+1}}{(n+1)!}\right|<\varepsilon^{\prime}.
\end{displaymath}
In particular, for $\varepsilon^{\prime}=\varepsilon^{n+1}/2(n+1)!$, we infer that the set $\#\lbrace \alpha_{j}^{\prime}<\varepsilon\rbrace$ is non-empty for all $f\in\Fcal_{\varepsilon^{\prime}}$. Consequently, $\alpha_{1}^{\prime}<\varepsilon$. Because $\alpha_{1}^{\prime}>0$, this concludes the proof.
\end{proof}
We refer the reader to \cite[Section 5]{Alberich} for a complementary discussion on the minimal spectral number and K. Saito's conjecture. In particular, a proof of  Proposition \ref{prop:necessary-conditions-Saito} \eqref{item:necessary-conditions-Saito-2} is also provided therein. 

\subsection{} For the geometric genus and the spectral genus, K. Saito's conjecture leads to the following expectations.
\begin{proposition}\label{prop:consequence-KSaito}
Assume that K. Saito's conjecture holds for some degenerating family of singularities $\Fcal$. Then:
\begin{enumerate}
  \item The geometric genus satisfies
  \begin{displaymath}
    \lim_{\Fcal}\frac{p_{g}}{\mu}=\frac{1}{(n+1)!}.
  \end{displaymath}
  \item \label{prop:conseq2} The spectral genus satisfies
\begin{equation}\label{eq:KSaito-and-our-conjecture}
    \lim_{\Fcal}\frac{\widetilde{p}_{g}}{\mu}= \frac{1}{(n+2)!}.
\end{equation}
    \item In particular, we have
\begin{displaymath}
    \lim_{\Fcal}\frac{\widetilde{p}_{g}}{p_{g}}=\frac{1}{n+2}.
\end{displaymath}
\end{enumerate}
\end{proposition}
\begin{proof}
The first item follows from \eqref{eq:Saito-limit-Hodge} in the case $n=p$, together with the evaluation of the volume of the standard simplex $A_{n}$.

For the second item, we first observe that
\begin{displaymath}
    \frac{\widetilde{p}_{g}}{\mu}=\int_{0}^{1}(1-s)\delta_{f}(s),
\end{displaymath}
according to \eqref{eq:spgenusMSaito}. Under K. Saito's conjecture, this converges to
\begin{displaymath}
    \int_{0}^{1}(1-s)N(s)ds=\int_{A_{n}}(1-x_{0}-\cdots-x_{n})dx_{0}\cdots dx_{n},
\end{displaymath}
Since the volume of $A_{n}$ is $1/(n+1)!$, by Lemma \ref{lemma:centermass} we conclude that the value of the integral is $1/(n+2)!$.  

The third point of the proposition is a combination of the first and the second points. 
\end{proof}
The second items of the proposition should be compared to the weak form of the conjecture: under both K. Saito's and our conjecture, the limit is approached from below. Proposition \ref{prop:conjecture-homogenenous} and Corollary \ref{cor:asymptotic} provide examples of this phenomenon. Also, notice that these statements are compatible with the known cases of K. Saito's conjecture reviewed in \textsection\ref{subsection:Saito-known-cases} above.

\subsection{Hertling's conjecture} While K. Saito's conjecture describes the distribution of the spectral numbers of an isolated singularity, Hertling's conjecture focuses on the variance. We discuss the relationship to our conjecture. Combined with K. Saito's conjecture, we conclude that Hertling's conjecture does not seem to trivially imply our strong conjecture for curves. 

\subsection{} Consider the spectral numbers of an isolated hypersurface singularity, given in the form $-1<\alpha_{1}\leq\ldots\leq\alpha_{\mu}\leq n$. Recall these are symmetric with respect to $\alpha\mapsto n-1-\alpha$. That is, we have the relationship $\alpha_{j}+\alpha_{\mu-j+1}=n-1$. Hence, the mean value of the spectral values is $(n-1)/2$. In \cite{Hertling}, Hertling proposed a bound for the variance. 

\begin{conjecture}[Hertling]
The spectral numbers of an isolated hypersurface singularity, taken in the interval $(-1,n)$, satisfy
\begin{equation}\label{eq:Hertling}
    \frac{1}{\mu}\sum_{j}\left(\alpha_{j}-\frac{n-1}{2}\right)^{2}\leq\frac{\alpha_{\mu}-\alpha_{1}}{12}.
\end{equation}
\end{conjecture}
Hertling's conjecture is known for quasi-homogeneous singularities. In this case, the predicted inequality is in fact an equality. This was proven by Hertling himself in \cite{Hertling}. An elementary proof was later found by Dimca \cite{Dimca:Hertling}. The case of curve singularities has been addressed by several authors. In the unpublished article \cite{MSaito:Hertling}, M. Saito proved the case of irreducible plane curve singularities. In \cite{Brelivet}, Br\'elivet settled the case of plane curves with convenient, non-degenerate Newton polygon. Later, in the unpublished article \cite{Brelivet:unpublished}, he considered general plane curve singularities.

\begin{theorem}[Br\'elivet, M. Saito]
Hertling's conjecture holds for isolated plane curve singularities. 
\end{theorem}
\qed

The previous statement has the following consequence for our conjecture. 
\begin{corollary}
For plane curve singularities, if $\alpha_{\mu}\leq \frac{2}{3}\sqrt{1-\mu^{-1}}$, then the strong form of the conjecture holds. 
\end{corollary}
\begin{proof}
For plane curve singularities, the spectral numbers normalized in $(-1,1)$ are invariant under the symmetry $x\mapsto -x$. In particular, $\alpha_{\mu}=-\alpha_{1}$. Hence, combining \eqref{eq:Hertling} with the Cauchy--Schwarz inequality, we derive
\begin{displaymath}
    \widetilde{p}_{g}=\sum_{0<\alpha_{j}<1}\alpha_{j}\leq \mu\sqrt{\frac{\alpha_{\mu}}{24}}.
\end{displaymath}
The claim is a straightforward consequence of the latter inequality.
\end{proof}

\subsection{} Suppose now that we are in the setting of a family $\Fcal$ of plane curve singularities for which K. Saito's conjecture holds. We can then apply Proposition \ref{prop:necessary-conditions-Saito} \eqref{item:necessary-conditions-Saito-2}. Recalling that $\alpha_{1}^{\prime}=\alpha_{1}+1$ and $\alpha_{1}=-\alpha_{\mu}$, we find
\begin{displaymath}
    \lim_{\Fcal}\alpha_{\mu}=1.
\end{displaymath}
Therefore, we see that the criterion provided by the corollary, which implies $\limsup_{\Fcal}\alpha_{\mu}\leq 2/3$, will apply at most to some exceptional singularities of the family $\Fcal$, but not in general. 

\section*{Funding}

The first author is supported by the Swedish Research Council, VR grant 2021-03838 "Mirror symmetry in genus one". The second author is supported by the Knut och Alice Wallenberg foundation "Guest researcher program".

\section*{Acknowledgements}

We thank Omid Amini, Jean-Benoît Bost, Lars Halvard Halle, Thomas Mordant, Claude Sabbah, Kyoji Saito and Jan Stevens for discussions relating to this article, and Josep \`Alvarez Montaner for bringing the article \cite{Alberich} to our attention. We heartily thank José Ignacio Burgos Gil for providing the suspension trick which became Lemma \ref{lemma:burgos}. We are deeply indebted to Ken-Ichi Yoshikawa for many discussions and the hospitality in Kyoto, and for many insights into the analytic torsion. The possibility of the continuity of the analytic torsion on Hilbert schemes was suggested by him. The second author wishes to thank the Department of Mathematics at Chalmers University of Technology and the University of Gothenburg for the hospitality too.

\bibliographystyle{amsplain}
\bibliography{Spectralgenus}{}

\end{document}